\documentclass[11pt]{amsart}
\usepackage{latexsym}
\usepackage{amsfonts}
\usepackage{amsmath,amssymb}
\usepackage{color}

\newcommand{\field}[1]{\mathbb{#1}}
\newcommand{\C}{\field{C}}

\newcommand{\id}[1]{\mathfrak{#1}}
\newcommand{\textswab}[1]{\id{#1}}

\newcommand{\ignore}[1]{}

\newtheoremstyle{s2}{9pt}{9pt}{\rm}{}{\bf}{.}{0.5em}{}
\theoremstyle{s2}
\newtheorem{defi}{Definition}[section]
\newtheorem{ex}[defi]{Example}
\newtheorem{re}[defi]{Remark}

\newtheoremstyle{s1}{9pt}{9pt}{\it}{}{\bf}{.}{0.5em}{}
\theoremstyle{s1}
\newtheorem{lem}[defi]{Lemma}
\newtheorem{theo}[defi]{Theorem}
\newtheorem{co}[defi]{Corollary}

\font\tenmsy=msbm10

\def\Bbb#1{\hbox{\tenmsy#1}}

\DeclareMathOperator{\rank}{rank}
\DeclareMathOperator{\Sing}{Sing}
\DeclareMathOperator{\mult}{mult}
\DeclareMathOperator{\codim}{codim}

\DeclareMathOperator{\Pic}{Pic}

\DeclareMathOperator{\Diag}{Diag}
\DeclareMathOperator{\diag}{diag}

\newcommand{\pa}[2]{\frac{\partial #1}{\partial #2}}

\title[Whitney theorem for complex polynomial mappings]{Whitney theorem for complex polynomial mappings} \makeatletter

\@addtoreset{equation}{section}
\makeatother
\author{M. \ Farnik \& Z. Jelonek \& M.A.S. Ruas}
\address[M. Farnik]{Jagiellonian University\\
Faculty of Mathematics and Computer Science\\
{\L}ojasiewi\-cza~6, 30-348 Krak\'ow, Poland}
\email{michal.farnik@gmail.com}

\address[Z. Jelonek]{Instytut Matematyczny\\
Polska Akademia Nauk\\
\'Sniadeckich 8, 00-656 Warszawa, Poland}
\email{najelone@cyf-kr.edu.pl}

\address[M.A.S. Ruas]{Departamento de Matem\'atica,
ICMC-USP, Caixa Postal 668, 13560-970 S\~ao Carlos, S.P., Brasil}
\email{maasruas@icmc.usp.br}


\subjclass{14 R 99, 32 A 10}
\thanks{The  authors are partially supported by the grant of Narodowe Centrum Nauki, grant number 2015/17/B/ST1/02637, additionally the third author is partially
supported by the FAPESP grant
2014/00304-2}

\begin{document}

\begin{abstract}
For given natural numbers $d_1,d_2$ 
we describe the topology of a generic polynomial mapping $F=(f,g)\colon X\to\C^2$, with deg $f\le d_1$ and deg $g\le d_2$. Here $X$ is a complex plane or a complex sphere.
\end{abstract}

\maketitle

\bibliographystyle{alpha}

\section{Introduction}
Polynomial mappings $F\colon \C^n\to\C^n$ are the most classical objects in the complex analysis, yet their topology has not been studied up till now. To the best knowledge of the authors complex algebraic families of polynomial mappings on affine varieties have not been investigated so far. Here we describe an idea of such study. We consider the family $\Omega_{\C^n}(d_1,\ldots,d_m)$ of polynomial mappings $F=(F_1,\ldots, F_m)\colon \C^n\to\C^m$ of degree bounded by $(d_1,\ldots,d_m)$.

 For a smooth affine variety $X \subset \C^n$ we also consider the family $\Omega_X(d_1,\ldots,d_m)=\{ F|_X: F\in \Omega_{\C^n}(d_1,\ldots,d_m)\}$. In particular based on Mather Projection Theorem, we prove  that  a generic member of $\Omega_X(d_1,\ldots,d_m)$ is transversal to  a given modular submanifold (in particular to a given Thom-Boardman strata $\Sigma^I$) in the space of multi-jets $_sJ^k(X,\C^m).$ Moreover, we show that a generic member of $\Omega_X(d_1,\ldots,d_m)$ is transversal to any smooth algebraic subvariety of the space of jets $J^k(X,\C^m)$, at least if $d_i\ge k.$ 

Let us recall that in \cite{jel2} the second author proved that if $M,X,Y$ are affine irreducible varieties, $X,Y$ are smooth  and $\Phi\colon M\times X \to Y$ is an algebraic family of polynomial mappings such that the generic element of this family is proper then two generic members of this family are topologically equivalent. In particular if $X\subset \C^p$ is of dimension $n$ and $m\ge n$ then any two generic members of the family $\Omega_X(d_1,\ldots,d_m)$ are topologically equivalent. For example, if $X$ is a smooth surface then the numbers $c_X(d_1,d_2)$ and $d_X(d_1,d_2)$ of cusps and double folds, respectively, of a generic member of the family $\Omega_X(d_1,d_2)$ are well-defined.

Our aim is to describe effectively the topology of such generic mappings. We consider in this paper the simplest case, when $n=m=2$ and $X=\C^2$ or $X$ is the complex sphere $S=\{ (x,y,z) \in \C^3: x^2+y^2+z^2=1\}$.
In these cases we describe the topology of the set $C(F)$  of critical points of $F$ and the topology of its discriminant $\Delta(F).$
In particular we show that a generic polynomial mapping $F\in \Omega_X(d_1,d_2)$ has only cusps, folds and double folds as singularities and we compute the number $c_X(d_1,d_2)$ of cusps and the number $d_X(d_1,d_2)$ of double folds of such generic polynomial mapping.
Our ideas work well also in higher dimensions. This paper is the first step in a study of the topology of
generic polynomial mappings $F\colon \C^n\to\C^n$.

The problem of counting the number of cusps of a
generic perturbation of a real plane-to-plane singularity was
considered by Fukuda and Ishikawa in \cite{fi}. They proved that
the number modulo 2 of cusps of a generic perturbation $F$ of an $\mathcal{A}$
finitely determined map-germ $F_0\colon (\mathbb R^{2},
0)\rightarrow(\mathbb R^{2},0)$ is a topological invariant of
$F_0$. More recently, in \cite{ks} Krzy\.zanowska and Szafraniec
gave an algorithm  to compute the number of cusps for sufficiently
generic fixed real polynomial mapping of the real plane.

Algebraic formulas to count the number of cusps and nodes of a generic
perturbation of an $\mathcal{A}$ finitely determined holomorphic map-germ $F_0\colon
(\mathbb C^2,0) \to (\mathbb C^2,0)$ were given by Gaffney and
Mond in \cite{gm1, gm2} (see also \cite{rieg}). In this case any two generic
perturbations $F$ of $F_0$ defined on a sufficiently small
neighborhood of $0$ are topologically equivalent, so the
numbers of cusps and nodes of $F$ are invariants of the map-germ $F_0$.

Let us note that in some cases our result allows also to use local methods to study global mappings. Indeed, in the special case when $\gcd(d_1,d_2)=1$ the numbers $c(F)$ and $d(F)$ can be computed by using local methods of Gaffney and Mond \cite{gm2} or Ohmoto methods \cite{ohm}  based on Thom polynomials. Note that in this case the leading homogenous part $F_h$ of a generic mapping $F=(f,g)$ is $\mathcal{A}$ finitely determined. Moreover,
we have a deformation $F_t(x)=(t^{d_1}f(t^{-1}(x)), t^{d_2}g(t^{-1}(x)))$.
Now we can use the fact (which is first proved in our paper) that a generic (with respect to the Zariski topology) mapping $F\in \Omega_X(d_1, d_2)$ has only folds, cusps and double folds as singularities.
Thus for the deformation $F_t\in \Omega_X(d_1, d_2)$ of $F$ all $F_t, \ t\not=0$ are generic mappings and all cusps and nodes of $F_t$ tend to $0$ when $t \rightarrow 0$.
In this case our formulas for $c(F)$ and $d(F)$ coincide with formulas of Gaffney-Mond etc.

However, in the general case these approaches do not work since any homogeneous mapping is not ${\mathcal A}$ finitely determined if $\gcd(d_1,d_2)\not=1$ (Gaffney-Mond, \cite{gm2}). Note that even if a germ $F$ is ${\mathcal K}$ finitely determined then in general the number of cusps and nodes depends on a given stable perturbation $F_t$ of $F$ (see Section \ref{finite}).
In particular in that case the local number of cusps or nodes cannot be defined and the methods of Gaffney-Mond and Ohmoto do not work. If $\gcd(d_1,d_2)\not=1$ our formulas do not coincide with formulas of Gaffney-Mond and Ohmoto, or rather the latter simply do not apply. Hence in general even discrete global invariants can not be obtained by local methods or methods based on Thom polynomials.

Now we will briefly describe the content of the paper.
In Section $2$ we state and prove general theorems. In Section $3$ we describe the topology of the set of critical points of a generic mapping
$F\in\Omega_{\C^2}(d_1,d_2)$. Moreover we compute the number $c_{\C^2}(d_1,d_2)$ of cusps. In Section $4$ we describe the topology of the discriminant
$\Delta(F)$ and compute the number $d_{\C^2}(d_1,d_2)$ of nodes of $\Delta(F)$. In Section $5$ we describe the topology of the set of critical points of a generic mapping $F\in\Omega_S(d_1,d_2)$, and compute the number $c_S(d_1,d_2)$, where $S\subset\C^3$ is a complex sphere. In Section $6$ we describe the topology of the discriminant $\Delta(F)$ and we compute the number $d_S(d_1,d_2)$.

In Section \ref{secGC} we introduce the notions of a generalized cusp and the index of a generalized cusp $\mu$ (see Definitions \ref{dfGenCus} and \ref{dfGenCusIn}). We show that if $F=(f, g)\colon X\to\C^2$ is an arbitrary polynomial mapping
with $\deg f\le d_1$, $\deg g\le d_2$ and generalized cusps at points $a_1,\ldots, a_r$ then $\sum^r_{i=1}
\mu_{a_i}\le c_X(d_1,d_2)$.

We conclude the paper with Section \ref{finite} which is devoted to proper stable deformations of a given polynomial mapping $F:X\to\C^m$.
In particular we give an example of a $\mathcal K$ finitely determined polynomial mapping $F:\C^2\to\C^2$ and its two stable deformations $F_t$, $G_t$ which have different number of cusps at $0$.

\section{General polynomial mappings}
Let $\Omega_{n}(d_1,\ldots,d_m)$ denote the space of polynomial mappings $F:\C^n\to\C^m$ of multi-degree bounded by $d_1,\ldots,d_m$.
Similarly if $X\subset \C^p$ is a smooth affine variety we consider the family $\Omega_{X}(d_1,\ldots,d_m)=\{ F|_X: F\in \Omega_{p}(d_1,\ldots,d_m)\}$.

By $J^q(\C^n,\C^m)$ we denote the space of $q$-jets of
polynomial mappings $F=(f_1,\ldots, f_m):\C^n \to \C^m$. We define
it exactly as in \cite{mather}.

If we fix coordinates in
the domain and the target then we can identify $J^q(\C^n,\C^m)$ with the
space $\C^n\times \C^m \times (\C^{N_q})^m$, where $\C^{N_q}$
parametrizes coefficients of polynomials of $n$-variables and of
degree bounded by $q$ with zero constant term (which correspond to
suitable Taylor polynomials). In further applications, in most
cases, we treat the space $J^q(\C^n,\C^m)$ in this simple way. In
particular for a given polynomial mapping $F\colon \C^n \to \C^m$ we
can define the mapping $j^q(F)$ as
$$j^q(F) \colon \C^n \ni x \mapsto \left(x, F(x),\left(\frac{\partial^{|\alpha|}{f_i}}{\partial x^\alpha}(x)\right)_{1\le
i\le m,1\le |\alpha|\le q}\right)\in J^q(\C^n,\C^m).$$

If $X^n\subset \C^p$ is a smooth affine variety then the space $J^q(X,\C^m)$ has the structure of a smooth algebraic manifold and can be locally represented in the same simple way as above. Indeed, locally $X$ is a complete intersection, i.e. for every point $x\in X$ there is an open neighborhood $U_x$ of $x$ such that $U_x= \{ g_1=0,\ldots, g_{p-n}=0\}$ (in some open set of $\C^p$) and $\rank \Big[\pa{g_i}{x_j}\Big] = p-n$ on $U_x$. We can assume that the mapping $(x_1,\ldots, x_n, g_1,\ldots, g_{p-n})$ is biholomorphic near $x$. In particular we have $x_i=\phi_i(x_1,\ldots,x_n)$ for $i>n$. Hence there exists another Zariski open neighborhood
$V_x$ of $x$ such that in $V_x$ we have global local holomorphic coordinates $x_1,\ldots,x_n$. In particular $J^q(V_x,\C^m)$ can be identified with the space $V_x\times \C^m \times (\C^{N_q})^m$. In local coordinates we have a mapping
$$j^q(F) \colon V_x \ni z \mapsto \left(z, F(z),\left(\frac{\partial^{|\alpha|}{f_i}}{\partial x^\alpha}(z)\right)_{1\le i\le m,1\le |\alpha|\le q}\right)\in J^q(V_x,\C^m).$$

Now we show that the space $J^q(X,\C^m)$ has the structure of a smooth algebraic manifold. Let $\mathcal D$ be a sheaf of derivations on $X$.
Since $\mathcal D$ is coherent  and $X$ is affine $\mathcal D$  is generated by a finite number of global sections $D_1,\ldots,D_s$. For a multi-index $\alpha=(\alpha_1,\ldots, \alpha_s)$ let $D^\alpha=D^{\alpha_1}\ldots D^{\alpha_s}$. Now let $Q$ be the number of multi-indexes $\alpha$
with $|\alpha|\leq k.$ Take $d_1=d_2\ldots=d_m=k$ and consider the mapping
$$\Psi: X\times \Omega_{n}(d_1,\ldots,d_m)\ni (x,F)\mapsto (x, F(x), (D^\alpha(x)(F|_X))_{|\alpha\leq k})\in X\times \C^m\times \C^Q.$$
It is easy to see that the mapping $\Psi$ is algebraic and its image is exactly the space $J^k(X,\C^m)$.

By $_sJ^q(X,\C^m)$ we denote the space of multi $q$-jets of polynomial mappings $F=(f_1,\ldots, f_n):X\to\C^m$. We denote by $\Diag$ the set $\{(x_1,\ldots x_s)\in X^s\ :\ x_i=x_j$ for some $i\neq j\}$ and for bundles $\pi_i:W_i\rightarrow X$ we denote by $\Diag_X$ the set $\{(w_1,\ldots w_s)\ :\ \pi_i(w_i)=\pi_j(w_j)$ for some $i\neq j\}$. We have  $_s{ J}^q(X,\C^m)=(J^q(X,\C^m))^s\setminus\Diag_X$. More generally, we define the space of $(q_1,\ldots,q_s)$-jets to be   $J^{q_1,\ldots,q_s}(X,\C^m):=J^{q_1}(X,\C^m)\times\ldots\times J^{q_s}(X,\C^m)\setminus\Diag_X$. We call them, if there is no danger of confusion, the space of multi-jets. Again, for a given polynomial mapping $F:X\to\C^m$ we have the mapping
$$J^{q_1,\ldots,q_s}(F) : X^s\setminus\Diag\mapsto(j^{q_1}(F)(x_1),\ldots,j^{q_s}(F)(x_s))\in  J^{q_1,\ldots,q_s}(X,\C^m).$$

In the sequel we use the Thom-Boardman manifolds $\Sigma^I$ (see \cite{boa}, \cite{math2}) which give stratifications of the jet space $J^k(X,\C^m)$. For a mapping $F:X\to\C^m$ we denote $\Sigma^I(F):=(J^q(F))^{-1}(\Sigma^I)$. The sets $\Sigma^i(F)$ consist of points where $F$ has corank exactly $i$. Moreover, if $\Sigma^{1_1,\ldots,i_k}(F)$ is a manifold then $\Sigma^{1_1,\ldots,i_k,i_{k+1}}(F)=\Sigma^{i_{k+1}}(F|\Sigma^{1_1,\ldots,i_k}(F))$.
We will also use the Thom-Boardman manifolds in the space of multi-jets. For bundles $\pi_i:W_i\rightarrow Y$ we denote by $\diag_Y$ the set $\{(w_1,\ldots w_s)\ :\ \pi_i(w_1)=\ldots=\pi_s(w_s)\}$. We denote $(\Sigma^{I_1},\ldots,\Sigma^{I_s}):=\Sigma^{I_1}\times\ldots\times\Sigma^{I_s}\cap\diag_Y$. 

Let us state the following result of Mather (this is an analogue of Theorem 1 in \cite{mather}, as Mather remarked after stating Theorem 6, the proof is analogous and the main change is the use of Bertini's theorem instead of Sard's theorem):

\begin{theo}
Let $X\subset \C^n$ be a smooth affine algebraic subvariety and let \linebreak $W\subset_sJ^q(X,\C^m)$ be a modular submanifold. There exists a Zariski open non-empty subset $U$ in the space of all linear mappings ${\mathcal L}(\C^n,\C^m)$ such that for every $L\in U$ the mapping $L: X\to \C^m$ is transversal $W$. 
\end{theo}

This theorem has the following nice application (which in the real smooth case was first observed by S. Ichiki in \cite{ichiki}):

\begin{co}\label{wazne}
Let $X\subset \C^n$ be an affine smooth algebraic subvariety, let $W\subset {_sJ^q(X,\C^m)}$ be a modular submanifold and let $F: X\to \C^m$ be a polynomial mapping. There exists a Zariski open non-empty subset $U$ in the space of all linear mappings 
${\mathcal L}(\C^n,\C^m)$ such that for every $L\in U$ the mapping $F+L: X\to \C^m$ is transversal to $W$. 
\end{co}

\begin{proof}
Let $G: X\ni x \mapsto (x, F(x))\in X\times \C^m$ and $\tilde{X}=graph(G)\cong X$. Apply Mather's theorem to the variety $\tilde{X}$.
We get that for a general matrix $A\in GL(m,m)$ and general linear mapping $L\in {\mathcal L}(\C^n,\C^m)$ the mapping $H(A,L)=A(F_1,\ldots,F_m)^T+L$ is transversal to $W$. Hence also the mapping $A^{-1}\circ H(A,L)$ is transversal to $W$ (because $W$ is invariant with respect to action of global biholomorphisms). This means that the mapping $F+ A^{-1}L$ is transversal to $W$. But we can specialize the matrix $A$ to the identity and the mapping $L$ to a given linear mapping $L_0\in {\mathcal L}(\C^n,\C^m)$. Hence we see that there is a dense subset of linear mappings $L\in {\mathcal L}(\C^n,\C^m)$ such that the mapping $F+L: X\to \C^m$ is transversal to $W$. However, the set of such mappings is a constructible subset of ${\mathcal L}(\C^n,\C^m)$. Since it is dense and constructible, it must contain a non-empty Zariski open subset. 
\end{proof}

We have the following general result which follows directly from Corollary \ref{wazne}:

\begin{theo}\label{trans}
Let $X\subset \C^n$ be a smooth algebraic variety and let $W\subset {_sJ^q(X,\C^m)}$ be a modular submanifold. Then there is a Zariski open subset  $V\subset \Omega_X(d_1,\ldots,d_m)$ such that for every $F\in V$ the mapping $F$ is
transversal to $W.$ In particular, there is a Zariski open subset  $U\subset \Omega_X(d_1,\ldots,d_m)$ such that for every $F\in U$ the mapping $F$ is
transversal to the Thom-Boardman strata $(\Sigma^{I_1},\ldots,\Sigma^{I_s})$ in $_sJ^q(X,\C^m)$.
\end{theo}

\begin{proof}
By Corollary \ref{wazne} the set of mappings $F\in\Omega_n(d_1,\ldots,d_m)$ which are transversal to $W$ is dense in $ \Omega_n(d_1,\ldots,d_m)$. However it is also constructible. Thus it must contain a Zariski open subset. Now observe that $(\Sigma^{I_1},\ldots,\Sigma^{I_s})$ is a modular
manifold.
\end{proof}

Note that Mather's result does not hold for every smooth subvariety in the space of jets, it requires the variety to be modular.
We prove here a result in the general direction -- we omit the assumption of modularity for the price of sufficiently high degree of the mapping.

We start with the following fact:

\begin{lem}\label{lem_interpolacja}
For every sequence of pairwise different points $c_1,\ldots,c_s\in \C^n$, a number $i\in \{1,\ldots,s\}$ and sequence of numbers $a_\alpha$, where $\alpha$ ranges through multiindexes $\alpha=(\alpha_{1},\ldots,\alpha_{n})$ with $0\le|\alpha|\le q_i$ there is a polynomial $H^i$ of degree bounded by $D\le \sum_{j=1}^m q_j+m-1$ for $i=1,\ldots,n$, such that:
\begin{enumerate}
\item for every multindex $\alpha$ with  $|\alpha|\le q_i$ we have $\frac{\partial^\alpha H^i}{\partial x_1^{\alpha_1}\ldots\partial x_n^{\alpha_n}}(c_i)=a_\alpha$, 
\item for every $j\not=i$ and every multindex $\beta$ with $|\beta|\le q_j$ we have 
$\frac{\partial^\beta H^i}{\partial x_1^{\beta_1}\ldots\partial x_n^{\beta_n}}(c_j)=0$.
\end{enumerate}
\end{lem} 

\begin{proof}
Using linear change of coordinates we can assume that $c_{i1}\not=c_{j1}$ for $i\not=j.$ By the symmetry it is enough to construct a polynomial $H^1$. Take 
 $$H^1_\alpha= (\sum_{|\alpha|\le q_i} b_\alpha (x-c_1)^\alpha)  \prod_{i=2}^m (x_1-c_{i1})^{q_i+1}.$$ We determine coefficients $b_\alpha$ inductively. If $\alpha=(0,...,0)=0$, then $b_0=a_0/\prod_{i=2}^m (c_{11}-c_{i1})^{q_i+1}.$ Now assume that we have all $b_\beta$ determined for 
 $|\beta|=k<q_1$ and we show how to determine $b_\alpha$ with $|\alpha|=k+1.$   We have $$\frac{\partial^{\alpha} H^1}{\partial x_1^{\alpha_1}...\partial x_n^{\alpha_n}}(x)=\alpha!b_\alpha \prod_{i=2}^m (x_1-c_{i1})^{q_i+1}+ R(x),$$ where $R(c_1)$ depends only at $c_1,...,c_n$ and $b_\gamma$ where $|\gamma|\le k.$ Hence it is enough to take $b_\alpha=(a_\alpha-R(c_1))/\alpha!\prod_{i=2}^m (c_{11}-c_{i1})^{q_i+1}.$
\end{proof}

Now we can prove:

\begin{theo}\label{thgeneric}
Let $X^n\subset \C^p$ be a smooth affine variety of dimension $n$. Let $S_1,\ldots, S_k$ be locally closed smooth algebraic submanifolds of $J^{q_1,\ldots,q_r}(X, \C^m)$.
Let $d_1,\ldots,d_m$ be integers such that $d_i\ge \sum_{j=1}^r q_j+r-1$ for $i=1,\ldots,m$.
Then there is a Zariski open dense subset
$U\subset \Omega_X(d_1,\ldots,d_m)$
such that for every $F\in U$ we
have $$j^{q_1,\ldots,q_r}(F|_X)\pitchfork S_i, \ for \ i=1,\ldots,k.$$
\end{theo}

\begin{proof}
First consider the case $X=\C^n$. For simplicity we can take $m=1$ (the general case is analogous). 
It is enough to prove that the mapping $\Omega_n(d_1)\times (\prod^r\C^n \setminus \Diag)\ni (F,x)\mapsto { J}^{q_1,\ldots,q_r}(F)(x)\in { J}^{q_1,\ldots,q_r}(\C^n,\C)$ is a submersion.
Let us observe that if we have a mapping $G: P\times Z \to P\times W$ of the form $G(p,z)=(p,g(p,z))\in P\times W$, then $G$ is a submersion if the mapping $Z\ni z\mapsto g(p,z)\in W$ is a submersion for every fixed $p\in P$. Now take $P=\prod^r \C^n\setminus\Diag$, $Z=\Omega_n(d_1)$ and $W$ in such way that $P\times W= { J}^{q_1,\ldots,q_m}(\C^n,\C)$. Put $G(p,F)=J^{q_1,\ldots,q_r}(F)(p)$. From the previous statement we have that $G$ is a submersion if $G(p,\cdot)$ is a submersion for every fixed $p$. But since $G(p,\cdot)$ is linear for fixed $p$ 
it is enough to prove that $G(p,\cdot)$ is surjective for fixed $p$. Hence our statement reduces to the Lemma \ref{lem_interpolacja}.

Now assume that $X$ is a general affine smooth variety. In generic system of linear coordinates, for a given points $a_1,...,a_r\in X$, we
can find a Zariski open subset $U$, which have global local coordinates
$x_{1},\ldots,x_{n}$, i.e.,  $x_{1},\ldots,x_{n}$ are holomorphic (local) coordinates in $U.$ 

Let $\Omega_n(d_1,\ldots,d_m)(x_{1},\ldots,x_{n})\subset \Omega_n(d_1,\ldots, d_m)$ denote the set of polynomial mappings, which depend only on variables $x_{1},\ldots,x_{n}.$ Note that we have $\Omega_n(d_1,\ldots,d_m)\cong \Omega_n(d_1,\ldots,d_m)(x_{1},\ldots,x_{n})\oplus W,$ where mappings in $W$ have coefficients  different from coefficients in
$\Omega_n(d_1,\ldots,d_m)$ $(x_{1},\ldots,x_{n})$. Note that
$W|_{U}$  is  the subset of holomorphic mappings locally depending only on variables $x_{1},\ldots,x_{n}$ and which have coefficients independent from coefficients in
$\Omega_n(d_1,\ldots,d_m)$ $(x_{1},\ldots,x_{n})$.
Now we can prove as above that $\Psi: \Omega_n(d_1,\ldots,d_m)\times U\ni
(F,x)\mapsto j^{q_1,\ldots,q_r}(F|_{U})(x)\in J^{q_1,\ldots,q_r}(U,\C^m)$ is a submersion (in the proof it is enough to use only parameters from $\Omega_n(d_1,\ldots,d_m)(x_{1},\ldots,x_{n})$.

Fix $1\le i\le k$. By the transversality theorem
with a parameter the set of polynomials $F\in
\Omega_n(d_1,\ldots,d_n)$ such that $j^{q_1,\ldots,q_r}(F|_X)$ is transversal to $S_i$
is dense in $\Omega_n(d_1,\ldots,d_m)$. On the other hand this set is
constructible in $\Omega_n(d_1,\ldots,d_m)$.

We conclude that there is a
Zariski open dense subset $V_i\subset \Omega_n(d_1,\ldots,d_m)$ such
that for every $F\in V_i$ we have $j^{q_1,\ldots,q_r}(F|_X)\pitchfork S_i$. Now it
is enough to take
$U=\bigcap^k_{i=1} V_i$.
\end{proof}

\begin{defi}
Let $\Sigma^k\subset J^1(X,\C^n)$ denote the subvariety of $1$-jets of
corank $k$. Let $F\in \Omega_X(d_1,\ldots,d_n)$. We say that
$F$ is one-generic if $F$ is proper and $j^1(F)\pitchfork \Sigma^1$.
\end{defi}

By Corollary \ref{wazne} the subset of {\it one-generic} mappings
contains a Zariski open dense subset of
$\Omega_X(d_1,\ldots,d_n)$. We have the following   result:

\begin{theo}\label{disc}
Let $X$ be a smooth complex manifold of dimension $n$. Let $F\colon X\to \C^n$ be a proper holomorphic one-generic mapping. Let $C(F)$
denote the set of critical points of $F$. Then there is an open
and dense subset $U\subset C(F)$ such that for every $a\in U$ the
germ $F_a\colon (X,a)\to (\C^n, F(a))$ is holomorphically equivalent
to a fold.
\end{theo}

\begin{proof}
Let $\Delta=F(C(F))$ be the discriminant of $F$. Take $U=C(F)\setminus
F^{-1}(\Sing(\Delta))$. The set $U$ is an open dense subset of
$C(F)$. Take a point $a\in U$ and consider the germ $F_a\colon
(X,a)\to (\C^n, F(a))$. By the choice of the point $a$ the germ
of the discriminant of $F_a$ is smooth. Hence by \cite{jel}, Corollary 1.11, the
germ $F_a$ is biholomorphically equivalent to a $k$-fold:
$(\C^n,0)\ni (x_1,\ldots, x_n)\mapsto (x_1^k,x_2,\ldots, x_n)\in
(\C^n,0)$. In particular ${\rm corank} [F_a]=1.$

Now note that
$J^1(\C^n,\C^n)\cong \C^n\times \C^n\times M(n,n)$, where $M(n,n)=\{
[a_{ij}], \ 1\le i,j\le n\}$ is the set of $n\times n$ matrices.
In these coordinates the set $\Sigma^1$ is given as $\{ (x,y,m) : \det
[m_{ij}]=\phi(x,y,m)=0\}$ on the open subset $\{ (x,y,m) :  {\rm corank} [m_{ij}]\le 1\}$. Since the
mapping $j^1(F)$ is transversal to $\Sigma^1$ the mapping $\phi\circ
j^1(F)=kx_1^{k-1}$ has to be a submersion at $0$. This is possible
only for $k=2$.
\end{proof}

\section{Plane mappings}\label{secCF}
Here we will study the set $\Omega_2(d_1,d_2)$.
Let us denote coordinates in $J^1(\C^2,\C^2)$ by
$$(x,y,f,g,f_x,f_y,g_x,g_y).$$ For a mapping
$F=(f,g)\in\Omega_2(d_1,d_2)$, we have $$j^1(F)=\left(x,y,f(x,y),g(x,y),
\frac{\partial{f}}{\partial{x}}(x,y),
\frac{\partial{f}}{\partial{y}}(x,y),
\frac{\partial{g}}{\partial{x}}(x,y),
\frac{\partial{g}}{\partial{y}}(x,y)\right),$$ which justifies our
notation. The set $\Sigma^1$ is given by the equation
$\phi(x,y,f,g,f_x,f_y,g_x,g_y)=f_xg_y-f_yg_x=0$. Since $\Sigma^1$
describes elements of rank one it is easy to see that it is a
smooth (non-closed) subvariety of $J^1(\C^2,\C^2)$.

Now we would like to describe the set $\Sigma^{1,1}$ effectively. We restrict our
attention only to sufficiently general jets. In the space
$J^2(\C^2,\C^2)$ we introduce coordinates
$$(x,y,f,g,f_x,f_y,g_x,g_y,f_{xx}, f_{yy}, f_{xy}, g_{xx}, g_{yy}, g_{xy}).$$
A generic mapping $F$ satisfies $\rank d_a F\ge 1$ for every $a$ (because
$\codim \Sigma^2=4$). We can assume that $F=(f,g)$ and $\nabla_a
f\not=0$. The critical set of $F$ is exactly the set $\Sigma^1(F)$
and it has a reduced equation
$\frac{\partial{f}}{\partial{x}}(x,y)\frac{\partial{g}}{\partial{y}}(x,y)-
\frac{\partial{f}}{\partial{y}}(x,y)\frac{\partial{g}}{\partial{x}}(x,y)=0$,
which we write for simplicity as $f_x g_y-f_y g_x=0$. In particular
the tangent line to $\Sigma^1(F)$ is given as
$$(f_{xx}g_y+f_xg_{xy}-f_{xy}g_x-f_yg_{xx})v+(f_{xy}g_y+f_xg_{yy}-f_{yy}g_x-f_yg_{xy})w=0.$$
Consequently the condition for $[F_a]\in \Sigma_{1,1}$ is:
$$f_xg_y-f_yg_x=0$$ and
$$(f_{xx}g_y+f_xg_{xy}-f_{xy}g_x-f_yg_{xx})f_y-(f_{xy}g_y+f_xg_{yy}-f_{yy}g_x-f_yg_{xy})f_x=0.$$
Let us note that the last equation contains terms $g_{xx}f_y^2$
and $g_{yy}f_x^2$ hence for $\nabla f\not=0$ these two equations
form a complete intersection. In general, if we omit the
assumption $\nabla f\not=0$ the set $\Sigma^{1,1}$ is given in $J^2(\C^2,\C^2)$
by three equations:
$$L_1:=f_xg_y-f_yg_x=0,$$
$$L_2:=(f_{xx}g_y+f_xg_{xy}-f_{xy}g_x-f_yg_{xx})f_y-(f_{xy}g_y+f_xg_{yy}-f_{yy}g_x-f_yg_{xy})f_x=0,$$
and
$$L_3:=(f_{xx}g_y+f_xg_{xy}-f_{xy}g_x-f_yg_{xx})g_y-(f_{xy}g_y+f_xg_{yy}-f_{yy}g_x-f_yg_{xy})g_x=0.$$
As above by symmetry the set $\Sigma^{1,1}$ is smooth and locally is
given as a complete intersection of either $L_1, L_2$ or $L_1,
L_3$.

We will denote by $J, J_{1,1}, J_{1,2}$ curves given by $L_1\circ
j^2(F)=0$, $L_2\circ j^2(F)=0$ and $L_3\circ j^2(F)=0$, respectively. We will also
identify these curves with their equations.

\begin{re}
{\rm These formulas give a description of $\Sigma^{1,1}$ also in the case of a general affine surface $X$, however, it might be only locally in the Zariski topology of $J^2(X,\C^2)$.}
\end{re}

\begin{defi}
Let $F\colon (\C^2,a)\to (\C^2,F(a))$ be a germ of a holomorphic mapping.  We say
that $F$ has a {\it fold} at $a$ if $F$ is biholomorphically
equivalent to the mapping $(\C^2,0)\ni (x,y)\mapsto (x, y^2)\in
(\C^2,0)$. Moreover, we say
that $F$ has a {\it cusp} at $a$ if $F$ is biholomorphically
equivalent to the mapping $(\C^2,0)\ni (x,y)\mapsto (x, y^3+xy)\in
(\C^2,0)$.
\end{defi}

\begin{re}
{\rm It is well known that $F$ has a fold at $a$ if $j^2(F)\pitchfork \Sigma^1$ at $a$ and $j^2(F)(a)\in\Sigma^{1,0}$, and $F$ has a cusp if $j^2(F)\pitchfork \Sigma^1, \Sigma^{1,1}$ and $j^2(F)\in\Sigma^{1,1}$.}
\end{re}

A direct consequence of Theorem \ref{trans} is:

\begin{theo}\label{cor1}(cf. \cite{wh})
Let $X\subset \C^n$ be a smooth algebraic surface, then there is a non-empty Zariski open subset $U\subset \Omega_X(d_1,d_2)$ such that for every $F\in U$ the mapping $F$ has only folds and cusps as singularities and the discriminant $F(C(F))$ has only cusps and nodes as singularities.
\end{theo}

Now we compute the number of cusps of a generic polynomial mapping
$F\in \Omega_2(d_1,d_2)$. To do this we need a series of lemmas:

\begin{lem}\label{infty0}
Let $L_\infty$ denote the line at infinity of
$\C^2$. There is a non-empty open subset $V\subset
\Omega_2(d_1,d_2)$ such that for all $(f,g)\in V:$

\begin{enumerate}
\item $\left\{\frac{\partial f}{\partial x}=0\right\}\pitchfork \left\{\frac{\partial
f}{\partial y}=0\right\}$, $\left\{\frac{\partial g}{\partial x}=0\right\}\pitchfork \left\{\frac{\partial
g}{\partial y}=0\right\}$,
\item $\overline{\left\{\frac{\partial f}{\partial x}=0\right\}}\cap \overline{\left\{\frac{\partial
f}{\partial y}=0\right\}}\cap L_\infty=\emptyset$, $\overline{\left\{\frac{\partial g}{\partial x}=0\right\}}\cap \overline{\left\{\frac{\partial
g}{\partial y}=0\right\}}\cap L_\infty=\emptyset$.
\end{enumerate}
\end{lem}

\begin{proof}
The case $d_1=1$ is trivial so assume $d_1>1$. Let us note that
the set $S\subset J^1(\C^2,\C^2)$ given by $\{ f_x=f_y=0\}$ is smooth.
Hence (1) follows from Theorem \ref{thgeneric}. To prove (2) it is
enough to assume that $f\in H_d$, where $H_d$ denotes the set of
homogenous polynomials of two variables of degree $d$.  Let $\Psi
\colon H_d\times (\C\times \C)\setminus \{0,0\}\ni (f, x,y)\mapsto
(\frac{\partial f}{\partial x}(x,y), \frac{\partial f}{\partial
y}(x,y))\in \C^2$. It is easy to see that $\Psi$ is a submersion.
Indeed, if $f=\sum a_i x^{d-i}y^i$ then $f_x:=\frac{\partial
f}{\partial x}(x,y)=da_0x^{d-1}+\ldots+a_{d-1}y^{d-1},
f_y:=\frac{\partial f}{\partial
y}(x,y)=a_1x^{d-1}+\ldots+da_dy^{d-1}$. Since $(x,y)\not=(0,0)$ we
can assume by symmetry that $y\not=0$. Now
$\pa{f_x}{a_{d-1}}=y^{d-1}, \pa{f_x}{a_d}=0,
\pa{f_y}{a_d}=dy^{d-1}$. Thus
$\pa{(f_x,f_y)}{(a_{d-1},a_d)}=dy^{2(d-1)}\not=0$.

Hence for a generic polynomial $f\in H_d$ the mapping $\Psi_f\colon
 (\C\times \C)\setminus \{(0,0)\}\ni (x,y)\mapsto (\frac{\partial f}{\partial
x}(x,y), \frac{\partial f}{\partial y}(x,y))\in \C^2$ is
transversal to the point $(0,0)$. In particular $\Psi_f^{-1}(0,0)$
is either zero-dimensional or the empty set. Since $f$ is a
homogenous polynomial the first possibility is excluded. This
means that  $\overline{\left\{\frac{\partial f}{\partial
x}=0\right\}}\cap \overline{\left\{\frac{\partial f}{\partial
y}=0\right\}}\cap L_\infty=\emptyset.$
\end{proof}

\begin{lem}\label{infty}
Let $L_\infty$ denote the line at infinity of
$\C^2$. There is a non-empty open subset $V\subset \Omega_2(d_1,d_2)$ such
that for all $F=(f,g)\in V$:

\begin{enumerate}
\item $\overline{J(F)}\cap \overline{J_{1,1}(F)}\cap L_\infty=\emptyset$,
\item $\overline{J(F)}\pitchfork L_\infty$.
\end{enumerate}
Here $\overline{J(F)}$ denotes the projective closure of the set $\{ J(F)=0\}$ etc.
\end{lem}

\begin{proof}
Since the case $d_1=d_2=1$ is trivial we may assume that $d_1>1$ or $d_2>1$.
We consider the (generic) case when $\deg f=d_1$ and $\deg g=d_2$.
Hence $\overline{J(F)}\cap L_\infty$ and $\overline{J_{1,1}(F)}\cap L_\infty$ depend only on the homogeneous parts of $f$ and $g$ of degree
$d_1$ and $d_2$ respectively. Let $H_d$ denote the set of
homogeneous polynomials of degree $d$ in two variables. It is
sufficient to show that there is an open subset $V\subset
H_{d_1,d_2}:=H_{d_1}\times H_{d_2}$ such that $\overline{J(F)}\cap \overline{J_{1,1}(F)}\cap L_\infty=\emptyset$ for all $F=(f,g)\in V$.

Consider the set $X=\left\{(p,F)\in\mathbb{P}^1\times H_{d_1,d_2}\
:\ J(F)(p)=J_{1,1}(F)(p)=0\right\}$. Note that $X$ is a
closed subset of $\mathbb{P}^1\times H_{d_1,d_2}$, and if
$\overline{J(F)}\cap \overline{J_{1,1}(F)}\cap L_\infty\neq\emptyset$ then
$F$ belongs to the image of the projection of $X$ on
$H_{d_1,d_2}$. So to prove (1) it is sufficient to show that $X$
has dimension strictly smaller than the dimension of
$H_{d_1,d_2}$.

Let $q=(1:0)\in \mathbb{P}^1$, $Y:=\{q\}\times
H_{d_1,d_2}$ and $X_0=X\cap Y$. Note that all fibers of the projection
$X\rightarrow\mathbb{P}^1$ are isomorphic to $X_0$. Thus
$\dim(X)=\dim(X_0)+\dim(\mathbb{P}^1)$ and to prove (1) it is sufficient to
show that $X_0$ has codimension at least $2$ in $Y$.

Let $(q,F)\in Y$ and let $a_i$ and $b_i$ be the parameters in $H_{d_1,d_2}$ giving
respectively the coefficients of $f$ at $x^{d_1-i}y^i$ and
of $g$ at $x^{d_2-i}y^i$. For $0\leq i+j\leq d_1$, we have
$\frac{\partial^{i+j}f}{\partial
x^iy^j}(q)=\frac{(d_1-j)!j!}{(d_1-i-j)!}a_j(F)$ and similarly
for $g$ and $b_j$.

To conclude the proof of (1) we will show that the codimension of $\{a_0b_0=0\}\cap X_0$ in $Y$ is at least $2$ and $\nabla J$ and $\nabla J_{1,1}$ are linearly independent outside $\{a_0b_0=0\}\cap X_0$ and thus the variety $X_0$ has
codimension $2$ in $Y$.

Let us calculate $J(p)$. We have $J(p)=(f_xg_y-f_yg_x)(q,F)=(d_1a_0b_1-d_2a_1b_0)(F)$.
Thus $\{a_0=0\}\cap X_0\subset \{a_0=a_1b_0=0\}\cap Y$ has codimension at least $2$ and we may assume in further calculations that $a_0(F)\neq 0$ and similarly $b_0(F)\neq 0$.

Let us assume that $d_2>1$. We have $\pa{J}{b_1}(p)=\pa{(d_1a_0b_1-d_2a_1b_0)}{b_1}(F)=d_1a_0(F)$ and $\frac{\partial J}{\partial b_2}(p)=0$.
Now let us calculate $\pa{J_{1,1}}{b_2}(p)$. The
coefficient $b_2$ can only be obtained from $\frac{\partial^2
g}{\partial y^2}$, which is present in $J_{1,1}$ in the
summand $-2\frac{\partial^2 g}{\partial
y^2}(d_1\pa{f}{x})^2$. Thus
$\pa{J_{1,1}}{b_2}(p)=\pa{(-2d_1^2b_2a_0^2)}{b_2}(F)=-2(d_1a_0(F))^2$.
So $\det\pa{(J,J_{1,1})}{(b_1,b_2)}(p)=-2(d_1a_0(F))^3\neq 0$.

Similarly, if $d_2=1$ and $d_1>1$ then $\det\pa{(J,J_{1,1})}{(a_1,a_2)}(p)=-2(d_1a_0(F))(d_2b_0(F))^2\neq 0$.

To prove (2) note that $\overline{\big\{\pa{J}{x}(F)=0\big\}}\cap\overline{\big\{\pa{J}{y}(F)=0\big\}}\subset \overline{J_{1,1}(F)}$, hence (1) implies (2).
\end{proof}

\begin{lem}\label{lemtrans1}
There is a non-empty open subset $V_1\subset \Omega_2(d_1,d_2)$
such that for all $(f,g)\in V_1$ and every $a\in \C^2$: if
$\frac{\partial{f}}{\partial{x}}(a)=0$ and
$\frac{\partial{f}}{\partial{y}}(a)=0$, then
$\frac{\partial{g}}{\partial{x}}(a)\not=0$ and
$\frac{\partial{g}}{\partial{y}}(a)\not=0$.
\end{lem}

\begin{proof}
Let us consider two subsets in $J^1(\C^2,\C^2)$: $R_1:=\{
(x,y,f,g,f_x,f_y,g_x,g_y): f_x=0,f_y=0, g_x=0\}$ and $R_2:=\{
(x,y,f,g,f_x,f_y,g_x,g_y): f_x=0,f_y=0, g_y=0\}$. By Theorem
\ref{thgeneric} there is a non-empty open subset $V_1\subset
\Omega_2(d_1,d_2)$ such that for every $F\in V_1$ the mapping
$j^1(F)$ is transversal to $R_1$ and $R_2$. Since these subsets
have codimension three, we see that the image of $j^1(F)$ is
disjoint with $R_1$ and $R_2$.
\end{proof}

\begin{lem}\label{lemtrans2}
There is a non-empty open subset $V_2\subset \Omega_2(d_1,d_2)$
such that for all $(f,g)\in V_2$ we have
$\big\{\frac{\partial{f}}{\partial{x}}=0\big\} \cap
\big\{\frac{\partial{f}}{\partial{y}}=0\big\}\cap J_{1,2}(f,g)=\emptyset$.
\end{lem}

\begin{proof}
Let us consider the (non-closed) subvariety $S\subset J^2(2)$
given by equations: $f_x=0$, $f_y=0$,
$(f_{xx}g_y+f_xg_{xy}-f_{xy}g_x-f_yg_{xx})g_y-(f_{xy}g_y+f_xg_{yy}-f_{yy}g_x-f_yg_{xy})g_x=0$,
$g_x\not=0$, $g_y\not=0$. It is easy to check that $S$ is a smooth
complete intersection and it has codimension three. The set of
generic mappings $F$ which are transversal to $S$ contains a
Zariski open dense subset $V_2\subset \Omega_2(d_1,d_2)$. By
construction for all $(f,g)\in V_2$ we have
$\big\{\frac{\partial{f}}{\partial{x}}=0\big\} \cap
\big\{\frac{\partial{f}}{\partial{y}}=0\big\}\cap J_{1,2}(f,g)=\emptyset$.
\end{proof}

\begin{lem}\label{lemtrans3}
There is a non-empty open subset $V_3\subset \Omega_2(d_1,d_2)$
such that for all $(f,g)\in V_3$ the curve $J(f,g)$ is transversal
to the curve $J_{1,1}(f,g)$.
\end{lem}

\begin{proof}
There is a Zariski open subset $V_3$ which contains only generic
mappings which satisfy hypotheses of all lemmas above. We can also
assume that the curves $\big\{\frac{\partial{f}}{\partial{x}}=0\big\}$ and
$\big\{\frac{\partial{f}}{\partial{y}}=0\big\}$ intersect transversally.
We have to show that the curves $J(f,g)$ and
 $J_{1,1}(f,g)$
intersect transversally at every point $a\in J(f,g)\cap
 J_{1,1}(f,g)$. If $\nabla_a
f\not=0$ then it follows from transversality of the mapping $F$
to the set $S_{1,1}$. Hence we can assume
$\frac{\partial{f}}{\partial{x}}(a)=0$ and
$\frac{\partial{f}}{\partial{y}}(a)=0$. By Lemma \ref{lemtrans1}
we have  $\frac{\partial{g}}{\partial{x}}(a)\not=0$ and
$\frac{\partial{g}}{\partial{y}}(a)\not=0$. Let us denote:
$\frac{\partial f}{\partial x}(x,y)=f_x$, $\frac{\partial
f}{\partial y}(x,y)=f_y$, etc.  It is enough to prove that in
the ring ${\mathcal O}_a^2$ we have the
 equality $I=(f_xg_y-f_yg_x, (f_{xx}g_y+f_xg_{xy}-f_{xy}g_x-f_yg_{xx})f_y-(f_{xy}g_y+f_xg_{yy}-f_{yy}g_x-f_yg_{xy})f_x)={\textswab m}_a$,
 where ${\textswab m}_a$ denotes the maximal ideal of  ${\mathcal O}_a^2$.
Put $L=f_xg_y-f_yg_x$. Hence $I=(L,L_xf_y-L_yf_x)$. Since
$g_x(a)\not=0$, $g_y(a)\not=0$, we have
$$I=(L,g_x[L_xf_y-L_yf_x],g_y[L_xf_y-L_yf_x])=(L,L_xg_xf_y-L_yg_xf_x,L_xg_yf_y-L_yg_yf_x)=$$
$$=(L,L_xg_yf_x-L_yg_xf_x,L_xg_yf_y-L_yg_xf_y)=(L,f_x[L_xg_y-L_yg_x],f_y[L_xg_y-L_yg_x]).$$
By Lemma \ref{lemtrans2} we have $[L_xg_y-L_yg_x](a)\not=0$, hence $I=(f_x,f_y)={\textswab m}_a$.
\end{proof}

Now we are in a position to prove:

\begin{theo}\label{thmcusps}
There is a Zariski open, dense subset $U\subset \Omega_2(d_1,d_2)$
such that for every mapping $F\in U$ the mapping $F$ has only
folds and cusps as singularities and the number of cusps is
equal to
$$d_1^2+d_2^2+3d_1d_2-6d_1-6d_2+7.$$
Moreover, if $d_1>1$ or $d_2>1$ then the set $C(F)$ of critical
points of $F$ is a smooth connected curve, which is topologically
equivalent to a sphere with $g=\frac{(d_1+d_2-3)(d_1+d_2-4)}{2}$
handles and $d_1+d_2-2$ points removed.
\end{theo}

\begin{proof} Note that by Theorem \ref{cor1} a generic $F$ has only folds and cusps as singularities.
 Note that every point $a$
of the intersection of curves $J(f,g)$ and $J_{1,1}(f,g)$ with
$\nabla_a f\not=0$ is a cusp. Moreover for a generic
mapping $F$ points with $\nabla_a f=0$ are not cusps (Lemma
\ref{lemtrans2}). By Bezout Theorem we have that in $J(f,g)\cap J_{1,1}(f,g)$ there are exactly
$(d_1-1)^2$ points with $\nabla f=0$ and
that the number of cusps of a generic mapping is equal to
$$(d_1+d_2-2)(2d_1+d_2-4)-(d_1-1)^2=d_1^2+d_2^2+3d_1d_2-6d_1-6d_2+7.$$

Finally by Lemma \ref{infty} we have that $C(F)=S_1(F)$ is a
smooth affine curve which is transversal to the line at infinity.
This means that $\overline{C(F)}$ is also smooth at infinity,
hence it is a smooth projective curve of degree $d=d_1+d_2-2$.
Thus by the Riemmann-Roch Theorem the curve $\overline{C(F)}$
has genus $g=\frac{(d-1)(d-2)}{2}$. This means in particular that
$\overline{C(F)}$ is homeomorphic to a sphere with
$g=\frac{(d-1)(d-2)}{2}$ handles. Moreover, by the Bezout Theorem
it has precisely $d$ points at infinity.
\end{proof}

\begin{re}
{\rm The curve $C(F)$ has $d_1+d_2-2$ (smooth) points at infinity and at each of these points it is transversal to the line at infinity. }
\end{re}

\section{The discriminant}\label{secDF}
Here we analyze the discriminant of a generic mapping from $\Omega(d_1,d_2)$. Let us recall that the discriminant of the mapping
$F\colon \C^2\to\C^2$ is the curve $\Delta(F):=F(C(F))$, where $C(F)$ is the critical curve of $F$. From Theorem \ref{trans} we have:

\begin{lem}\label{bir}
There is a non-empty open subset $U\subset \Omega_2(d_1,d_2)$
such that for every mapping $F\in U$:
\begin{enumerate}
\item $F_{|C(F)}$ is injective outside a finite set,
\item if $p\in\Delta(F)$ then $|F^{-1}(p)\cap C(F)|\leq 2$,
\item if $|F^{-1}(p)\cap C(F)|= 2$ then the curve $\Delta(F)$ has a normal
crossing at $p$.
\end{enumerate}
\end{lem}
\begin{proof}
By Theorem \ref{trans} can find a $U$ such that we have the required transversality to Thom-Boardman strata. (1) and (3) follow from transversality to $(\Sigma^1,\Sigma^1)$ and (2) follows from transversality to $(\Sigma^1,\Sigma^1,\Sigma^1)$.
\end{proof}

Hence for a generic $F$ the only singularities of $\Delta(F)$ are cusps and nodes. We showed in Theorem \ref{thmcusps} that there are exactly $c(F)=d_1^2+d_2^2+3d_1d_2-6d_1-6d_2+7$ cusps. Now we will compute the number $d(F)$ of nodes of $\Delta(F)$. We will use the following theorem of Serre (see \cite{mil}, p. 85):

\begin{theo}\label{thmgenusdelta}
If $\Gamma$ is an irreducible curve of degree $d$ and genus $g$  in the complex projective plane
then $$\frac{1}{2} (d-1)(d-2)= g + \sum_{z\in \Sing(\Gamma)} \delta_z,$$
where $\delta_z$ denotes the delta invariant of a point $z$.
\end{theo}

First we compute the degree of the discriminant:

\begin{lem}\label{lemdegdisc}
Let $F=(f,g)\in \Omega(d_1,d_2)$ be a generic mapping. If $d_1\ge d_2$ then $\deg\Delta(F)= d_1(d_1+d_2-2)$.
\end{lem}

\begin{proof}
Let $L\subset \C^2$ be a generic line $\{ax+by+c=0\}$. Then $L$ intersects $\Delta(F)$ in smooth points and $\deg\Delta(F)=\# L\cap \Delta(F)$.
If $j\colon C(F)\to \Delta(F)$ is a mapping induced by $F$ then $\# L\cap \Delta(F)=\# j^{-1}(L\cap \Delta(F)).$ The curve $j^{-1}(L)=\{af+bg+c=0\}$ has no common points at infinity
with $C(F)$. Hence by Bezout Theorem we have $\# j^{-1}(L\cap \Delta(F))=(\deg j^{-1}(L))(\deg C(F))=d_1(d_1+d_2-2)$. Consequently $\deg \Delta(F) = d_1(d_1+d_2-2)$.
\end{proof}

We have the following method of computing the delta invariant (see \cite{mil}, p. 92-93):

\begin{theo}\label{milnor2}
Let $V_0\subset \C^2$ be an irreducible germ of an analytic curve with the Puiseux parametrization of the form
 $$z_1=t^{a_0},\ z_2=\sum_{i>0} \lambda_i t^{a_i}, \text{ where }\lambda_i\neq 0,\ a_1<a_2<a_3<\ldots$$
Let $D_j=\gcd(a_0,a_1,\ldots, a_{j-1}).$ Then $$\delta_0=\frac{1}{2}\sum_{j\ge 1} (a_j-1)(D_j-D_{j+1}).$$
If $V=\bigcup^r_{i=1} V_i$ has $r$ branches then $$\delta(V)=\sum^r_{i=1} \delta(V_i)+\sum_{i<j} V_i\cdot V_j,$$
where $V\cdot W$ denotes the intersection product.
\end{theo}

The main result of this section will be based on the following:

\begin{theo}\label{theodeltaz}
Let $F\in \Omega(d_1,d_2)$ be a generic mapping. Let $d_1\geq d_2$ and $d=\gcd(d_1,d_2)$. Denote by $\overline{\Delta}$ the projective closure
of the discriminant $\Delta$. Then $$\sum_{z\in (\overline{\Delta}\setminus \Delta)} \delta_z= \frac{1}{2}d_1(d_1-d_2)(d_1+d_2-2)^2+\frac{1}{2}(-2d_1+d_2+d)(d_1+d_2-2).$$
\end{theo}

\begin{proof}
Let $\tilde{f}(x,y,z)=z^{d_1}f\left(\frac{x}{z},\frac{y}{z}\right)$ and $\tilde{g}(x,y,z)=z^{d_2}g\left(\frac{x}{z},\frac{y}{z}\right)$ be the homogenizations of $f$ and $g$, respectively,  and let $\overline{f}(x,z)=\tilde{f}(x,1,z)$ and $\overline{g}(x,z)=\tilde{g}(x,1,z)$. For a generic mapping the curves $C(F)$ and $\{f=0\}$ have no common points at infinity (see Lemma \ref{lemstyczne}). Moreover we may assume that $(1:0:0)\notin \overline{C(F)}$. Thus $F$ extends to a neighborhood of $\overline{C(F)}\cap L_\infty$ on which it is given by the formula
$$\overline{F}(x,z)= \left(z^{d_1-d_2}\frac{\overline{g}(x,z)}{\overline{f}(x,z)}, \frac{z^{d_1}}{\overline{f}(x,z)}\right).$$

Let  $\{P_1,\ldots,P_{d_1+d_2-2}\}=\overline{C(F)}\cap L_\infty$, fix a point $P=P_i$. The curve $\overline{C(F)}$ is transversal to the line at infinity so it has a local parametrization at $P$ of the form $\gamma(t):=(\sum_i e_it^i,t)$. We have the following:

\begin{lem}\label{lemstyczne}
If $F$ is a generic mapping then $\overline{f}(P)\neq 0$, $\overline{g}(P)\neq 0$ and
$$\overline{f}(\gamma(t))=\overline{f}(P)(1+ct+\ldots),\ \overline{g}(\gamma(t))=\overline{g}(P)(1+dt+\ldots),$$
where $cd\neq 0$ and $d_2c\neq d_1d$.
\end{lem}

\begin{proof}
Let $\tilde{J}=\tilde{J}(F)$ be the homogenization of $J(F)$. Obviously $\tilde{J}=\pa{\tilde{f}}{x}\pa{\tilde{g}}{y}-\pa{\tilde{f}}{y}\pa{\tilde{g}}{x}$.
Let $\overline{J}(x,z)=\tilde{J}(x,1,z)$. Since $\overline{J}(\gamma(t))=0$ and $\pa{\gamma(t)}{t}_{|t=0}=(e_1,1)$ we have
$$\overline{f}(P)c\pa{\overline{J}(P)}{x}=\left(\pa{\overline{f}(\gamma(t))}{t}\pa{\overline{J}(\gamma(t))}{x}\right)_{|t=0}=$$
$$\left(\pa{\overline{f}(\gamma(t))}{x}a_1\pa{\overline{J}(\gamma(t))}{x}+\pa{\overline{f}(\gamma(t))}{z}\pa{\overline{J}(\gamma(t))}{x}\right)_{|t=0}=$$
$$\left(-\pa{\overline{f}(\gamma(t))}{x}\pa{\overline{J}(\gamma(t))}{z}+\pa{\overline{f}(\gamma(t))}{z}\pa{\overline{J}(\gamma(t))}{x}\right)_{|t=0}=
\pa{\overline{f}(P)}{z}\pa{\overline{J}(P)}{x}- \pa{\overline{f}(P)}{x}\pa{\overline{J}(P)}{z}.$$

Consider the set
$$X=\left\{(p,F)\in L_\infty\times\Omega_2(d_1,d_2):\ \tilde{J}(F)(p)=
\left(\pa{\tilde{f}}{z}\pa{\tilde{J}(F)}{x}-\pa{\tilde{f}}{x}\pa{\tilde{J}(F)}{z}\right)(p)=0\right\}.$$

Note that if $\overline{f}(P)=0$ or $c=0$ then the fiber over $F$ of the projection from $X$ to $\Omega_2(d_1,d_2)$ is non-empty. Hence it suffices to prove that $X$ has codimension at least $2$.

Let $p=(0:1:0)$ and $q=(a:b:0)\in L_\infty\setminus \{(1:0:0)\}$. Let $\tilde{T}(x,y,z)=(bx-ay,y,z)$ so that $\tilde{T}(q)=p$. Take $T(x,y)=(bx-ay,y).$ Note that $\tilde{J}(F\circ T)=(\tilde{J}(F)\circ \tilde{T})J(\tilde{T})=b\tilde{J}(F)\circ \tilde{T}$. Furthermore
$$\pa{\tilde{f}\circ \tilde{T}}{z}\pa{\tilde{J}(F\circ T)}{x}-\pa{\tilde{f}\circ T}{x}\pa{\tilde{J}(F\circ T)}{z}=b^2\left(\pa{\tilde{f}}{z}\pa{\tilde{J}(F)}{x} -\pa{\tilde{f}}{x}\pa{\tilde{J}(F)}{z}\right)\circ \tilde{T}.$$
Thus $(p,F)\mapsto(T^{-1}(p),F\circ T)$ is an isomorphism of $X_p:=X\cap(\{p\}\times\Omega_2(d_1,d_2))$ and $X\cap(\{q\}\times\Omega_2(d_1,d_2))$. So it is enough to show that $X_p$ has codimension $2$ in $Y_p:=\{p\}\times\Omega_2(d_1,d_2)$.

Let $a_i$ be the parameters in $\Omega_2(d_1,d_2)$ giving the coefficients of $\tilde{f}$ (and of $f$) at $x^{d_1-i}y^i$ and let $b_i$ and $c_i$ describe respectively the coefficients of $\tilde{g}$ at $x^{d_2-i}y^i$ and $x^{d_2-i-1}y^iz$.

The first equation of $X_p$ is $d_2a_{d_1-1}b_{d_2}-d_1a_{d_1}b_{d_2-1}=0$ and the only summand of the second containing $c_{d_2-1}$ is $-(a_{d_1-1})^2(d_2-1)c_{d_2-1}$. Clearly those equations are independent outside the set $\{a_{d_1-1}=0\}$. Moreover $\{a_{d_1-1}=d_2a_{d_1-1}b_{d_2}-d_1a_{d_1}b_{d_2-1}=0\}=\{a_{d_1-1}=a_{d_1}=0\}\cup\{a_{d_1-1}=b_{d_2-1}=0\}$, thus $X_p$ has codimension 2 in $Y_p$.

Finally note that if $d_2c=d_1d$ then $$d_2\overline{g}(P)\left(\pa{\overline{f}}{z}\pa{\overline{J}}{x}-\pa{\overline{f}}{x}\pa{\overline{J}}{z}\right)(P)=
d_1\overline{f}(P)\left(\pa{\overline{g}}{z}\pa{\overline{J}}{x}-\pa{\overline{g}}{x}\pa{\overline{J}}{z}\right)(P).$$ Hence we consider the set
$$Z=\Big\{(p,F)\in L_\infty\times\Omega_2(d_1,d_2):\ \tilde{J}(F)(p)=$$
$$d_2\tilde{g}(p)\left(\pa{\tilde{f}}{z}\pa{\tilde{J}}{x}-\pa{\tilde{f}}{x}\pa{\tilde{J}}{z}\right)(p)-
d_1\tilde{f}(p)\left(\pa{\tilde{g}}{z}\pa{\tilde{J}}{x}-\pa{\tilde{g}}{x}\pa{\tilde{J}}{z}\right)(p)=0\Big\}.$$
Similarly as above one can show that it has codimension $2$, which concludes the proof.
\end{proof}

Let $C_p$ be the branch of $\overline{C(F)}$ at $P$. We find the Puiseux expansion of the branch $\overline{F}(C_P)$ of $\overline{\Delta(F)}$ at $\overline{F}(P)$.  We have
$$\overline{F}(\gamma(t))=\left(t^{d_1-d_2}\frac{\overline{g}(\gamma(t))}{\overline{f}(\gamma(t))},\frac{t^{d_1}}{\overline{f}(\gamma(t))}\right)=$$
$$\left(t^{d_1-d_2}(1+(d-c)t+\ldots)\frac{\overline{g}(P)}{\overline{f}(P)},\frac{t^{d_1}(1-ct+\ldots)}{\overline{f}(P)}\right).$$

If $d_1=d_2$ then by Lemma \ref{lemstyczne} we have $d-c\neq 0$ and $\overline{F}(C_P)$ is smooth at $\overline{F}(P)$. So assume $d_1>d_2$. Since the function $h(t)=\left(\frac{\overline{f}(P)}{\overline{g}(P)}\frac{\overline{g}(\gamma(t))}{\overline{f}(\gamma(t))}\right)^{\frac{1}{d_1-d_2}}=
1+\frac{d-c}{d_1-d_2}t+\ldots$ is invertible in $t=0$ we can introduce a new variable $T=th(t)$. We have $\overline{F}(\gamma(T))=\left(T^{d_1-d_2}\frac{\overline{g}(P)}{\overline{f}(P)},T^{d_1}h(t)^{-d_1}(1-ct+\ldots)\frac{1}{\overline{f}(P)}\right)$.
Moreover $$h(t)^{-d_1}(1-ct+\ldots)=\left(1-d_1\frac{d-c}{d_1-d_2}T+\ldots\right)(1-cT+\ldots)=1+\frac{d_2c-d_1d}{d_1-d_2}T+\ldots$$ By Lemma \ref{lemstyczne} we have $d_2c-d_1d\neq 0$ and we can apply Theorem \ref{milnor2} to compute $\delta(\overline{F}(C_P))_{\overline{F}(P)}$. Since $a_0=d_1-d_2$, $a_1=d_1$ and $a_2=d_1+1$, we have $2\delta(\overline{F}(C_P))_{\overline{F}(P)}=(d_1-1)(d_1-d_2-d)+(d_1+1-1)(d-1)=(d_1-1)(d_1-d_2-1)+(d-1)$, where $d=\gcd(d_1,d_2)$.

To proceed further we also need:

\begin{lem}\label{row}
If $F$ is a generic mapping then
$$\overline{f}(P_i)^{d_2}\overline{g}(P_j)^{d_1}\neq\overline{f}(P_j)^{d_2}\overline{g}(P_i)^{d_1}$$
for $i,j\in\{1,2,\ldots,d_1+d_2-2\}$ and $i\neq j$.
\end{lem}

\begin{proof}
Consider the set $X=\{(p,q,F)\in L_\infty\times L_\infty\times\Omega_2(d_1,d_2):\ p\neq q,\ \tilde{J}(F)(p)=\tilde{J}(F)(q)=
\tilde{f}(p)^{d_2}\tilde{g}(q)^{d_1}-\tilde{f}(q)^{d_2}\tilde{g}(p)^{d_1}=0\}$. Similarly as in Lemma \ref{lemstyczne} we will prove that $X$ has codimension $3$, so there is a dense open subset $S\subset \Omega(d_1,d_2)$ such that the projection from $X$ has empty fibers over $F\in S$.

Indeed, take $p=(1:0:0)$, $q=(0:1:0)$ and $Y:=\{(p,q)\}\times\Omega_2(d_1,d_2)$. It suffices to show that $X_0=X\cap Y$ has codimension $3$ in $Y$. Let $a_i$ and $b_i$ be the parameters in $\Omega_2(d_1,d_2)$ giving respectively the coefficients of $\tilde{f}$ at $x^{d_1-i}y^i$ and of $\tilde{g}$ at $x^{d_2-i}y^i$.

The three equations describing $X_0$ are $w_1=d_1a_0b_1-d_2a_1b_0=0$, $w_2=d_2a_{d_1-1}b_{d_2}-d_1a_{d_1}b_{d_2-1}=0$ and $w_3=a_0^{d_2}b_{d_2}^{d_1}-a_{d_1}^{d_2}b_0^{d_1}=0$. Note that $X_0\cap\{a_0=0\}=\{a_0=b_0=w_2=0\}\cup\{a_0=a_1=a_{d_1}=w_2=0\}$ has codimension $3$. Similarly $X_0\cap\{b_0=0\}$ and $X_0\cap\{a_{d_1}=0\}$ have codimension $3$, however outside the set $\{a_0=b_0=a_{d_1}\}$ the three equations are obviously independent. Thus $X_0$ has codimension $3$ in $X$.
\end{proof}

Now we are in a position to compute $\sum_{z\in (\overline{\Delta}\setminus \Delta)} \delta_z$. If $d_1=d_2$ then $\overline{\Delta}$ has exactly $d_1+d_2-2$ smooth points at infinity and consequently $\sum_{z\in (\overline{\Delta}\setminus \Delta)} \delta_z=0$ (see the text after the proof of Lemma \ref{lemstyczne}). So assume $d_1>d_2$, then $\overline{\Delta}$ has only one point at infinity $Q=(1:0:0)$.
In $Q$ the curve $\overline{\Delta}$ has exactly $r=d_1+d_2-2$ branches $V_i=\overline{F}(C_{P_i})$. We computed above that $2\delta(V_i)_Q=(d_1-1)(d_1-d_2-1)+(d-1)$. Now we will compute $V_i\cdot V_j$.
Let $t_{a,b}(x,y)=(x+a,y+b)$. By the dynamical definition of intersection there exists a neighborhood $U$ of $0$, such that for small generic $a,b$ we have
$$V_i\cdot V_j=\# (U\cap V_i\cap t_{a,b}(V_j)).$$
This means that $V_i\cdot V_j$ is equal to the number of solutions of the following system:
$$\frac{\overline{g}(P_i)}{\overline{f}(P_i)}T^{d_1-d_2}=\frac{\overline{g}(P_j)}{\overline{f}(P_j)}S^{d_1-d_2}+a,$$
$$\frac{1}{\overline{f}(P_i)}T^{d_1}(1+\alpha_iT+\ldots)=\frac{1}{\overline{f}(P_j)}S^{d_1}(1+\alpha_jS+\ldots)+b,$$
where $a,b$ and $S,T$ are sufficiently small. Take
$$Q:(\C^2,0)\to (\C^2,0),$$
$$Q(T,S)=\left(\frac{\overline{g}(P_i)}{\overline{f}(P_i)}T^{d_1-d_2}-\frac{\overline{g}(P_j)}{\overline{f}(P_j)}S^{d_1-d_2},
\frac{1}{\overline{f}(P_i)}T^{d_1}(1+\alpha_iT+\ldots)-\frac{1}{\overline{f}(P_j)}S^{d_1}(1+\alpha_jS+\ldots)\right).$$
Thus we have $V_i\cdot V_j=\mult_0 Q$. Note that by Lemma \ref{row} the minimal homogenous polynomials of the two components of $Q$ have no nontrivial common zeroes, hence $V_i\cdot V_j=d_1(d_1-d_2)$. Consequently
$$\sum_i \delta(V_i)+ \sum_{i>j} V_i\cdot V_j=\frac{1}{2}[(d_1-1)(d_1-d_2-1)+(d-1)](d_1+d_2-2)+$$
$$\frac{1}{2}d_1(d_1-d_2)(d_1+d_2-2)(d_1+d_2-3)=$$
$$\frac{1}{2}d_1(d_1-d_2)(d_1+d_2-2)^2+\frac{1}{2}(-2d_1+d_2+d)(d_1+d_2-2).$$
\end{proof}

We can now prove the following:

\begin{theo}
There is a Zariski open, dense subset $U\subset \Omega_2(d_1,d_2)$
such that for every mapping $F\in U$ the discriminant $\Delta(F)=F(C(F))$ has only cusps and nodes as singularities. Let $d=\gcd(d_1,d_2)$.
Then the number of cusps is equal to
$$c(F)=d_1^2+d_2^2+3d_1d_2-6d_1-6d_2+7$$
and the number of nodes is equal to
$$d(F)=\frac{1}{2}\left[(d_1d_2-4)((d_1+d_2-2)^2-2)-(d-5)(d_1+d_2-2)-6\right].$$
\end{theo}

\begin{proof}
Let $d_1\geq d_2$ and $D=d_1+d_2-2$. By Lemma \ref{lemdegdisc} we have $\deg\Delta(F)= d_1D$. From Lemma \ref{bir} we know that $\Delta(F)$ has only cusps and nodes as singularities and is birational with $C(F)$. Hence $\Delta(F)$ has genus $g=\frac{1}{2}(D-1)(D-2)$. Thus by Theorem \ref{thmgenusdelta} we have
$$\frac{1}{2} (d_1D-1)(d_1D-2)=\frac{1}{2}(D-1)(D-2)+c(F)+d(F)+ \sum_{z\in (\overline{\Delta}\setminus \Delta)} \delta_z.$$
 Substituting
$$\sum_{z\in (\overline{\Delta}\setminus \Delta)} \delta_z=\frac{1}{2}d_1(d_1-d_2)D^2+\frac{1}{2}(-2d_1+d_2+d)D$$
from Theorem \ref{theodeltaz} we obtain
$$2(c(F)+d(F))=d_1d_2D^2-D^2+3D-d_1D-d_2D-dD=(d_1d_2-2)D^2-(d-1)D.$$
Thus by Theorem \ref{thmcusps} we get:
$$d(F)=\frac{1}{2}\left[(d_1d_2-2)D^2-(d-1)D-2(D^2-2D+d_1d_2-1)\right]=$$
$$\frac{1}{2}\left[(d_1d_2-4)(D^2-2)-(d-5)D-6\right].$$
\end{proof}

\begin{re}
{\rm If $d_1=d_2=d$ then the discriminant has $2d-2$ smooth points at infinity and at each of these points it is tangent to the line $L_\infty$ (at infinity) with multiplicity $d.$
If $d_1>d_2$ then the discriminant has only one point at infinity with $d_1+d_2-2$ branches $V_1,\ldots, V_{d_1+d_2-2}$ and each of these branches has delta invariant
$$\delta(V_i)=\frac{(d_1-1)(d_1-d_2-1)+(\gcd(d_1,d_2)-1)}{2}$$ and $V_i\cdot L_\infty=d_1.$ Additionally $V_i\cdot V_j=d_1(d_1-d_2)$. In particular the branches $V_i$ are smooth if and only if
$d_1=d_2$ or $d_1=d_2+1.$}
\end{re}

\section{The complex sphere}
In the next two sections we show that our method can be easily generalized to the case when $X$ is a complex sphere.
Let $\phi=y^2+2xz$ and let $S$ be a complex sphere: $S=\{ (x,y,z): \phi=1\}$ (of course $S$ is linearly equivalent with a standard sphere $S':=\{ (x,y,z): x^2+y^2+z^2=1\}$).
Here we will study the set $\Omega_S(d_1,d_2)$. We consider on the set $\Omega_S(d_1,d_2)$ the Zariski topology, which is the induced  topology given by the mapping $\Theta : \Omega_3(d_1,d_2)\ni F\mapsto F|_S\in \Omega_S(d_1,d_2).$ 

First we compute the critical set $C(F)$ of a generic mapping $F=(f,g)\in \Omega_S(d_1,d_2)$.
Note that $x\in C(F)$ if rank $(\nabla \phi, \nabla f, \nabla g)<3,$  hence $C(F)$ is the intersection of $S$ and the surface given by
$$J(F)= \left|\begin{matrix} z & y & x \\ f_x & f_y & f_z \\ g_x & g_y & g_z \end{matrix}\right| = 0.$$

In particular we have:

\begin{co}
For a generic mapping $F\in \Omega_S(d_1,d_2)$ we have $\deg C(F)=2(d_1+d_2-1).$
\end{co}

Now we describe cusps of a generic mapping $F\colon S\to \C^2$. Note that a tangent line to $C(F)$ is given by two equations:
$$ zv_1+yv_2+xv_3=0, \ \ J(F)_xv_1+J(F)_yv_2+J(F)_zv_3=0.$$

The mapping $F$ has a cusp in a point $(x,y,z)$ if

(1) $(x,y,z)\in C(F)$

(2) the line given by the kernel of  $d_{(x,y,z)} F$ is tangent to $C(F).$

First let us determine the kernel of $d_{(x,y,z)} F$. If $\rank \left|\begin{matrix} z & y & x \\ f_x & f_y & f_z \end{matrix}\right| = 2$ then
the kernel is given by the vector
$$v(f)=\left(\left|\begin{matrix} y & x \\  f_y & f_z \end{matrix}\right|,-\left|\begin{matrix} z  & x \\ f_x & f_z \end{matrix}\right|,\left|\begin{matrix} z & y \\ f_x & f_y  \end{matrix}\right|\right).$$ Otherwise it is the vector $$v(g)=\left(\left|\begin{matrix} y & x \\  g_y & g_z \end{matrix}\right|,-\left|\begin{matrix} z  & x \\ g_x & g_z \end{matrix}\right|,\left|\begin{matrix} z & y \\ g_x & g_y  \end{matrix}\right|\right).$$

Let $J_{1,1}(F):=J(F)_x v_1(f)+J(F)_y v_2(f)+J(F)_z v_3(f)$ and $J_{1,2}(F):=J(F)_x v_1(g)+J(F)_y v_2(g)+J(F)_z v_3(g)$. Let $C$ denote the set of cusps of $F$, for generic $F$ we have from the construction:
$$ C=\{J(F)=J_{1,1}(F)=J_{1,2}(F)=0\}.$$

Furthermore, we will show in Lemma \ref{lemS1} that $S\cap \{J_{1,2}(F)=0\}\cap \{v(f)=0\}=\emptyset$ which gives
$$C=S\cap (\{J(F)=J_{1,1}(F)=0\}\setminus \{v(f)=0\}).$$

\begin{lem}\label{lemS1}
Let $L_\infty$ denote the plane at infinity of
$\C^3$. There is a non-empty open subset $V\subset \Omega_S(d_1,d_2)$ such
that for all $F=(f,g)\in V$:

\begin{enumerate}
\item $S\cap \{J_{1,2}(F)=0\}\cap \{v(f)=0\}=\emptyset$, $S\cap \{J_{1,1}(F)=0\}\cap \{v(g)=0\}=\emptyset$,
\item $\overline{S}\cap \overline{\{J(F)=0\}}\cap \overline{\{J_{1,1}(F)=0\}}\cap L_\infty=\emptyset$, $\overline{S}\cap \overline{\{J(F)=0\}}\cap \overline{\{J_{1,2}(F)=0\}}\cap L_\infty=\emptyset$,
\item $\overline{S} \cap \overline{\{J(F)=0\}}\pitchfork L_\infty$.
\end{enumerate}
\end{lem}

\begin{proof}
(1) The assertion can be proved locally. Consider the open set $U_z=\{ p\in S: z\not=0\}$ (and similarly open sets $U_x,U_y$). In $U_z$ we have globally defined local coordinates $x,y$. Now the proof reduces to Lemma \ref{lemtrans2}.

(2)  Similarly as in Lemma \ref{infty} we will show that there is an open subset $V\subset H_{d_1,d_2}:=H_{d_1}\times H_{d_2}$ such that $\overline{S}\cap \overline{\{J(F)=0\}}\cap \overline{\{J_{1,1}(F)=0\}}\cap L_\infty=\emptyset$ for all $F=(f,g)\in V$.
Let $\phi(x,y,z)=y^2+2xz$ and $\Gamma:=\{ (x,y,z)\in \Bbb P^2: \phi(x,y,z)=0\}$. Obviously $\Gamma\cong \Bbb P^1$.

Consider the set $X=\left\{(p,F)\in\Gamma\times H_{d_1,d_2}\ :\ \phi(p)=J(F)(p)=J_{1,1}(F)(p)=0\right\}$.
If $\overline{\{\phi=0\}}\cap \overline{\{J(F)=0\}}\cap \overline{\{J_{1,1}(F)=0\}}\cap L_\infty\neq\emptyset$ then
$F$ belongs to the image of the projection of $X$ on
$H_{d_1,d_2}$. So to prove (1) it is sufficient to show that $X$
has dimension strictly smaller than the dimension of $H_{d_1,d_2}$.

Let $q=(1:0:0)\in \mathbb{P}^2$, $Y:=\{q\}\times
H_{d_1,d_2}$ and $X_0=X\cap Y$. Note that all fibers of the projection
$X\rightarrow\Gamma$ are isomorphic to $X_0$, because the group $GL(S)$ of linear transformations of $S$ acts transitively on the conic at infinity of $S$. Thus $\dim(X)=\dim(X_0)+\dim(\Gamma)$ and to prove (1) it is sufficient to
show that $X_0$ has codimension at least $2$ in $Y$.

Let $r=(q,F)\in Y$ and let $a_{i,j}$ and $b_{i,j}$ be the parameters in $H_{d_1,d_2}$ giving
respectively the coefficients of $f$ at $x^{d_1-i-j}y^iz^j$ and
of $g$ at $x^{d_2-i-j}y^iz^j$. For $0\leq i+j+k\leq d_1$, we have
$\frac{\partial^{i+j+k}f}{\partial
x^iy^jz^k}(q)=\frac{(d_1-j-k)!j!k!}{(d_1-i-j-k)!}a_{j,k}(F)$ and similarly
for $g$ and $b_{j,k}$.

To conclude the proof of (1) we will show that the codimension of $\{a_{1,0}b_{1,0}=0\}\cap X_0$ in $Y$ is at least $2$ and $\nabla J$ and $\nabla J_{1,1}$ are linearly independent outside $\{a_{1,0}b_{1,0}=0\}\cap X_0$ and thus the variety $X_0$ has
codimension $2$ in $Y$.

Let us calculate $J(r)$. We have $J(r)=(f_xg_y-f_yg_x)(q,F)=(d_1a_{0,0}b_{1,0}-d_2a_{1,0}b_{0,0})(F)$.
Thus $\{a_{0,0}=0\}\cap X_0\subset \{a_{0,0}=a_{1,0}b_{0,0}=0\}\cap Y$ has codimension at least $2$ and in further calculations we may assume that $a_{0,0}(F)\neq 0$ and similarly $b_{0,0}(F)\neq 0$.

Let us assume that $d_2>1$. We have $\pa{J}{b_{1,0}}(r)=\pa{(d_1a_{0,0}b_{1,0}-d_2a_{1,0}b_{0,0})}{b_{1,0}}(F)=d_1a_{0,0}(F)$ and $\frac{\partial J(r)}{\partial b_{2,0}}=0$.
Now let us calculate $\pa{J_{1,1}}{b_{2,0}}(r)$. The
coefficient $b_{2,0}$ can only be obtained from $\frac{\partial^2
g}{\partial y^2}$, which is present in $J_{1,1}$ in the
summand $\frac{\partial^2 g}{\partial
y^2}\left|\begin{matrix} z & x  \\ f_x & f_z  \end{matrix}\right|^2.$ Thus
$\pa{J_{1,1}}{b_{2,0}}(p)=\pa{(2b_{2,0}d_1^2a_{0,0}^2)}{b_{2,0}}(F)=2d_1^2a_{0,0}(F)^2$.
So $\det\pa{(J,J_{1,1})}{(b_{1,0},b_{2,0})}(p)=2d_1^3(a_{0,0}(F))^3\neq 0$.

Similarly, if $d_2=1$ and $d_1>1$ then $\det\pa{(J,J_{1,1})}{(a_{0,1},a_{0,2})}(p)=2d_2^3(b_{1,0}(F))^3\neq 0$.

(3) Note that $\overline{\big\{\nabla J(F)|_S=0\big\}}\subset \{\overline{J_{1,1}(F)=0\}}$, hence (2) implies (3).
\end{proof}

\begin{lem}\label{lemS2}
There is a non-empty open subset $V_1\subset \Omega_S(d_1,d_2)$
such that for all $(f,g)\in V_1$ the curve $S\cap J(f,g)$ is transversal
to the curve $S\cap J_{1,1}(f,g)$.
\end{lem}

\begin{proof}
As in Lemma \ref{lemS1} (1) we consider the sets $U_x,U_y,U_z$ with globally defined local coordinates and reduce the proof to Lemmas \ref{lemtrans1} and \ref{lemtrans3}.
\end{proof}

\begin{lem}\label{infty5'}
There is a non-empty open subset $V_2\subset H_{d_1}$
such that for all $f\in V_2$ the equations:
\begin{enumerate}
\item $\phi(x,y,z)=0,$
\item $v(f)=0$
\end{enumerate}
\noindent have no common solutions different from $(0,0,0)$.
\end{lem}

\begin{proof}
We proceed similarly as in Lemma \ref{lemS1} (2).

Let $\Gamma:=\{ (x,y,z)\in \Bbb P^2: \phi(x,y,z)=0\}\cong \Bbb P^1$. Consider the set $$X=\left\{(p,f)\in\Gamma\times H_{d_1}\ :\ \phi(p)=v_1(f)(p)=v_2(f)(p)=v_3(f)(p)=0\right\}.$$
If ${\{\phi=0\}}\cap \{ v(f)=0\}\neq\emptyset$ then
$f$ belongs to the image of the projection of $X$ on
$H_{d_1}$. So to prove (1) it is sufficient to show that $X$
has dimension strictly smaller than the dimension of
$H_{d_1}$.

Let $q=(1:0:0)\in \mathbb{P}^2$, $Y:=\{q\}\times
H_{d_1}$ and $X_0=X\cap Y$. As before, all fibers of the projection
$X\rightarrow\Gamma$ are isomorphic to $X_0$, so
$\dim(X)=\dim(X_0)+\dim(\Gamma)$ and it is sufficient to
show that $X_0$ has codimension at least $2$ in $Y$.

But $X_0$ is given by two equations: $-a_{(1,0)}=0, d_1a_{(0,0)}=0$, so $\codim X_0=2$.
\end{proof}

\begin{lem}\label{infty6'}
There is a non-empty open subset $V_3\subset \Omega_S(d_1,d_2)$
such that for all $(f,g)\in V_3$ the equations:
\begin{enumerate}
\item $y^2+2xz=1,$
\item $v(f)=0$
\end{enumerate}
\noindent have exactly $2(d_1^2-d_1+1)$ common solutions.
\end{lem}

\begin{proof}
We have $$v(f)=\left(\left|\begin{matrix} y & x \\  f_y & f_z \end{matrix}\right|,-\left|\begin{matrix} z  & x \\ f_x & f_z \end{matrix}\right|,\left|\begin{matrix} z & y \\ f_x & f_y  \end{matrix}\right|\right).$$
Note that generically the curve $\left\{\left|\begin{matrix} y & x \\  f_y & f_z \end{matrix}\right|=0\right\}\cap \left\{\left|\begin{matrix} z  & x \\ f_x & f_z \end{matrix}\right|=0\right\}$ decomposes into $\{v(f)=0\}$ and $\{ x=f_z=0 \}$. Thus by the Bezout Theorem $\deg\{v(f)=0\}=d_1^2-d_1+1$ and $S\cap\{v(f)=0\}$ has $2(d_1^2-d_1+1)$ points. We leave checking that the intersections are transversal and there are no components at infinity to the reader.
\end{proof}

Now we are in a position to prove:

\begin{theo}\label{thmcusps'}
There is a Zariski open, dense subset $U\subset \Omega_S(d_1,d_2)$
such that for every mapping $F=(f,g)\in U$ the mapping $F$ has only
folds and cusps as singularities and the number of cusps is
equal to
$$2(d_1^2+d_2^2+3d_1d_2-3d_1-3d_2+1).$$
Moreover the set $C(F)$ of critical
points of $F$ is a smooth connected curve, which is topologically
equivalent to a sphere with $(d_1+d_2-2)^2$
handles and $2(d_1+d_2-1)$ points removed.
\end{theo}

\begin{proof}

Note that every point $a$
of the intersection of curves $J(f,g)$ and $J_{1,1}(f,g)$ with
$v(f)\not=0$ is a cusp. Moreover for a generic
mapping $F$ points with $v(f)=0$ are not cusps (Lemma
\ref{lemS1}). By Lemma \ref{infty6'} we have that in the set $S\cap\{ v(f)=0\}$ there are exactly
$2(d_1^2-d_1+1)$ points  and
that the number of cusps of a generic mapping is equal to
$$2[(d_1+d_2-1)(2d_1+d_2-2)-(d_1^2-d_1+1)]=2(d_1^2+d_2^2+3d_1d_2-3d_1-3d_2+1).$$

Moreover by Lemma \ref{lemS1} we have that $C(F)=S_1(F)$ is a
smooth affine curve which is transversal to the plane at infinity.
This means that $J:=\overline{C(F)}$ is also smooth at infinity,
hence it is a smooth projective curve of degree $2(d_1+d_2-1)$.
Note that $\Pic(\overline{S})=\Bbb Z L_1\oplus \Bbb Z L_2$, where
$L_1,L_2$ are suitable lines in $\overline{S}$ (for details see e.g. \cite{szaf},  Ex.2 p. 237). Moreover if $H$ is a plane section then $H\sim L_1+L_2$.
Hence in $\Pic(\overline{S})$ we have $\overline{C(F)}\sim aL_1+bL_2$ where $a+b=2(d_1+d_2-1)$.

Take $l_i=L_i\cap S$ and note that $\Pic(S)$ is generated freely by $l_1$ or $l_2$ with the relation $l_1+l_2=0$.
In particular $C(F)\sim (a-b)l_1$. But in $\Pic(S)$ we have $C(F)\sim (d_1+d_2-1)H=0$.
Thus $a=b=d_1+d_2-1$.

Suppose that $C(F)$ is not connected. Hence $\overline{C(F)}=\Gamma_1+\Gamma_2$.
We have $\Gamma_1 \sim a_1 L_1+b_1 L_2$ and $\Gamma_2\sim a_2 L_1+ b_2 L_2$, where $a_1,b_1,a_2,b_2 \ge 0$, $a_1+b_1>0$ and $a_2+b_2>0$.
Note that $a_1+a_2=b_1+b_2=d_1+d_2-1>0$ thus if $a_1b_2=0$ then $a_2b_1>0$. So $\Gamma_1.\Gamma_2=a_1b_2+a_2b_1>0$. Consequently
$\Gamma_1\cap \Gamma_2\neq \emptyset$ and $\overline{C(F)}$ is not smooth -- a contradiction. This implies that $C(F)$ is connected.

Let $H\subset \Bbb P^3$ be a hyperplane. The canonical divisor of $\overline{S}$ is
$-2H=-2(L_1+L_2)$. Hence $K_{J}=(J-2H)|_J=(d_1+d_2-3)(L_1+L_2)|_J$ and $\deg K_J=2(d_1+d_2-3)(d_1+d_2-1)$.
By Riemmann-Roch Theorem $J$ has genus $\deg K_J/2+1=(d_1+d_2-2)^2$.
This means in particular that
$\overline{C(F)}$ is homeomorphic to a sphere with
$(d_1+d_2-2)^2$ handles. Moreover, by the Bezout Theorem
it has precisely $2(d_1+d_2-1)$ points at infinity.

\end{proof}

\begin{re}
{\rm The curve $C(F)$ has $2(d_1+d_2-1)$ (smooth) points at infinity and in each of these points it is transversal to the plane at infinity.}
\end{re}

\section{The complex sphere: the discriminant}\label{secDF'}
Here we analyze the discriminant of a generic mapping from $\Omega_S(d_1,d_2)$. Similarly as for the plane Theorem \ref{trans} implies that for a generic $F$ the only singularities of $\Delta(F)$ are cusps and nodes. We showed in Theorem \ref{thmcusps'} that there are exactly $c(F)=2(d_1^2+d_2^2+3d_1d_2-3d_1-3d_2+1)$ cusps. Now we will compute the number $d(F)$ of nodes of $\Delta(F)$.
First we compute the degree of the discriminant:

\begin{lem}\label{lemdegdisc'}
Let $F=(f,g)\in \Omega_S(d_1,d_2)$ be a generic mapping. If $d_1\ge d_2$ then $\deg\Delta(F)= 2d_1(d_1+d_2-1)$.
\end{lem}

\begin{proof}
Since the proof is analogous to the proof of Lemma \ref{lemdegdisc} we skip it.
\end{proof}

The main result of this section will be based on the following:

\begin{theo}\label{theodeltaz'}
Let $F\in \Omega_S(d_1,d_2)$ be a generic mapping. Let $d_1\geq d_2$ and $d=\gcd(d_1,d_2)$. Denote by $\overline{\Delta}$ the projective closure
of the discriminant $\Delta$. Then $$\sum_{z\in (\overline{\Delta}\setminus \Delta)} \delta_z= 2d_1(d_1-d_2)(d_1+d_2-1)^2+(-2d_1+d_2+d)(d_1+d_2-1).$$
\end{theo}

\begin{proof}
Let $\tilde{f}(x,y,z,w)=w^{d_1}f\left(\frac{x}{w},\frac{y}{w}, \frac{z}{w}\right)$ and $\tilde{g}(x,y,z,w)=w^{d_2}g\left(\frac{x}{w},\frac{y}{w}, \frac{z}{w}\right)$ be the homogenizations of $f$ and $g$ and let $\overline{f}(x,y,w)=\tilde{f}(x,y,1,w)$ and $\overline{g}(x,y,w)=\tilde{g}(x,y,1,w)$, respectively. For a generic mapping the curves $C(F)$ and $\{f=0\}$ have no common points at infinity (see Lemma \ref{lemstyczne'}). Moreover since $F$ is generic, we have  $\{z=0\}\cap \overline{C(F)}=\emptyset$. Thus $F$ extends to a neighborhood of $\overline{C(F)}\cap L_\infty$ on which it is given by the formula
$$\overline{F}(x,y,w)= \left(w^{d_1-d_2}\frac{\overline{g}(x,y,w)}{\overline{f}(x,y,w)}, \frac{w^{d_1}}{\overline{f}(x,y,w)}\right).$$

Let $\Gamma=\overline{S}\cap L_\infty$. Let $\{P_1,\ldots,P_{2d_1+2d_2-2}\}=\overline{C(F)}\cap\Gamma$, fix a point $P=P_i$. The curve $\overline{C(F)}$ is transversal to the line at infinity so it has a local parametrization at $P$ of the form $\gamma(t):=(\sum_i a_it^i,\sum_i b_it^i,t)$. We have the following:

\begin{lem}\label{lemstyczne'}
If $F$ is a generic mapping then $\overline{f}(P)\neq 0$, $\overline{g}(P)\neq 0$ and
$$\overline{f}(\gamma(t))=\overline{f}(P)(1+ct+\ldots),\ \overline{g}(\gamma(t))=\overline{g}(P)(1+dt+\ldots),$$
where $cd\neq 0$ and $d_2c\neq d_1d$.
\end{lem}

\begin{proof}
Let  $\tilde{J}$ be the homogenization of $J$. Obviously $$\tilde{J}(F) = \left|\begin{matrix} z & y & x \\ \tilde{f}_x & \tilde{f}_y & \tilde{f}_z \\ \tilde{g}_x & \tilde{g}_y & \tilde{g}_z \end{matrix}\right|.$$

Now let $\overline{J}(x,y,w)=\tilde{J}(x,y,1,w)$ and $\psi(x,y,w)=2x+y^2-w^2=\tilde{\phi}(x,y,1,w)$, where $\tilde{\phi}$ is the homogenization of $\phi=y^2+2xz-1$. We have $\overline{J}(\gamma(t))=0$ and $\psi(\gamma(t))=0$. Moreover,
$\pa{\gamma(t)}{t}_{|t=0}=(a_1,b_1,1)$. Thus we have
$$\pa{\psi(P)}{x}a_1+\pa{\psi(P)}{y}b_1+\pa{\psi}{w}(P)=0,$$
$$\pa{\overline{J}(P)}{x}a_1+\pa{\overline{J}(P)}{y}b_1+\pa{\overline{J}(P)}{w}=0.$$

Consequently $a_1=\overline{a_1}\delta^{-1}$ and $b_1=\overline{b_1}\delta^{-1}$, where
$$\overline{a_1}=\pa{\psi(P)}{w}\pa{\overline{J}(P)}{y}-\pa{\psi(P)}{y}\pa{\overline{J}(P)}{w},\ \overline{b_1}=\pa{\psi(P)}{x}\pa{\overline{J}(P)}{w}-\pa{\psi(P)}{w}\pa{\overline{J}(P)}{x},$$ $$\delta=\pa{\psi(P)}{x}\pa{\overline{J}(P)}{y}-\pa{\psi(P)}{y}\pa{\overline{J}(P)}{x}.$$

Thus
$$\overline{f}(P)c\delta=\pa{\overline{f}(P)}{x}\overline{a_1}+\pa{\overline{f}(P)}{y}\overline{b_1}+\pa{\overline{f}(P)}{w}\delta.$$

Take
$$\tilde{a_1}=\pa{\tilde{\psi}(P)}{w}\pa{\tilde{J}(P)}{y}-\pa{\tilde{\psi}(P)}{y}\pa{\tilde{J}(P)}{w}, \ \tilde{b_1}=\pa{\tilde{\psi}(P)}{x}\pa{\tilde{J}(P)}{w}-\pa{\tilde{\psi}(P)}{w}\pa{\tilde{J}(P)}{x},$$ $$\tilde{\delta}=\pa{\tilde{\psi}(P)}{x}\pa{\tilde{J}(P)}{y}-\pa{\tilde{\psi}(P)}{y}\pa{\tilde{J}(P)}{x}.$$

Consider the set
$$X=\left\{(P,F)\in \Gamma\times\Omega_3(d_1,d_2):\ \tilde{J}(F)(p)=\pa{\tilde{f}(P)}{x}\tilde{a_1}+\pa{\tilde{f}(P)}{y}\tilde{b_1}+\pa{\tilde{f}(P)}{w}\tilde{\delta}=0
\right\}.$$

Note that if $\overline{f}(P)=0$ or $c=0$ then the fiber over $F$ of the projection from $X$ to $\Omega_3(d_1,d_2)$ is non-empty. Hence it suffices to prove that $X$ has codimension at least $2$.

Let $p=(0:0:1:0)$, and $q=(-a^2/2: a:1:0)\in \Gamma$. Let $\tilde{T}(x,y,z,w)=(x+ay-a^2z/2,y-az,z,w)$ and $T(x,y,z)=(x+ay-a^2z/2,y-az,z).$ Thus $T(S)=S$ and $\tilde{T}(q)=p$. As in Lemma \ref{lemstyczne} we can show that $(p,F)\mapsto(\tilde{T}^{-1}(p),F\circ T)$ is an isomorphism of $X_p:=X\cap(\{p\}\times\Omega_3(d_1,d_2))$ and $X\cap(\{q\}\times\Omega_3(d_1,d_2))$. So it is enough to show that $X_p$ has codimension $2$ in $Y_p:=\{p\}\times\Omega_3(d_1,d_2)$.

Let $a_{i,j,k}$ be the parameters in $\Omega_3(d_1,d_2)$ giving the coefficients of $\tilde{f}$ at $x^iy^jz^{d_1-i-j-k}w^k$ (i.e. of $f$ at $x^iy^jz^{d_1-i-j-k}$) and let $b_{i,j,k}$ describe the analogous coefficients of $\tilde{g}$.

The first equation of $X_p$ is $w_1:=d_2a_{0,1,0}b_{0,0,0}-d_1b_{0,1,0}a_{0,0,0}$. The second one is
$$w_2=a_{1,0,0}\tilde{a}_1+a_{0,1,0}\tilde{b}_1+a_{0,0,1}\tilde{\delta}=a_{0,1,0}\pa{\tilde{J}(p)}{w}+a_{0,0,1}\pa{\tilde{J}(p)}{y}=$$
$$a_{0,1,0}((d_2-1)b_{0,0,1}a_{0,1,0}+ d_2b_{0,0,0}a_{0,1,1}-d_1a_{0,0,0}b_{0,1,1}-(d_1-1)a_{0,0,1}b_{0,1,0})+$$ 
$$a_{0,0,1}(a_{1,0,0}b_{0,0,1}-a_{0,0,1}b_{1,0,0}+ 
2d_2a_{0,2,0}b_{0,0,0}-2d_1a_{0,0,0}b_{0,2,0}+(d_2-1)a_{0,1,0}b_{0,1,0}-(d_1-1)a_{0,1,0}b_{0,1,0)}).$$
By direct computation we obtain
$$\pa{w_1}{b_{0,0,0}}=d_2a_{0,1,0},  \pa{w_2}{a_{0,2,0}}=2d_2a_{0,0,1}b_{0,0,0};\pa{w_1}{a_{0,0,0}}=-d_1b_{0,1,0},  \pa{w_2}{b_{0,2,0}}=-2d_1b_{0,0,1}a_{0,0,0}.$$
Thus  the equations $w_1=0$ and $w_2=0$ are independent outside the set $$\{  a_{0,1,0}=0\}\cap \{  b_{0,1,0}=0\}\cup (\{ a_{0,0,1}=0\}\cup \{b_{0,0,0}=0\})\cap ( \{ b_{0,0,1}=0\}\cup \{ a_{0,0,0}=0\}). $$ So $X_p$ has codimension 2 in $Y_p$.
Finally note that if $d_2c=d_1d$ then $$d_2\overline{g}(P)\left(\pa{\tilde{f}(P)}{x}\tilde{a_1}+\pa{\tilde{f}(P)}{y}\tilde{b_1}+\pa{\tilde{f}(P)}{w}\tilde{\delta}\right)=
d_1\overline{f}(P)\left(\pa{\tilde{g}(P)}{x}\tilde{a_1}+\pa{\tilde{g}(P)}{y}\tilde{b_1}+\pa{\tilde{g}(P)}{w}\tilde{\delta}\right).$$ Hence we consider the set
$$Z=\Big\{(p,F)\in \Gamma\times\Omega_3(d_1,d_2):\ \tilde{J}(F)(p)=$$
$$d_2\overline{g}(P)\left(\pa{\tilde{f}(P)}{x}\tilde{a_1}+\pa{\tilde{f}(P)}{y}\tilde{b_1}+\pa{\tilde{f}(P)}{w}\tilde{\delta}\right)=
d_1\overline{f}(P)\left(\pa{\tilde{g}(P)}{x}\tilde{a_1}+\pa{\tilde{g}(P)}{y}\tilde{b_1}+\pa{\tilde{g}(P)}{w}\tilde{\delta}\right)=0\Big\}.$$
Similarly as above one can show that it has codimension $2$, which concludes the proof.
\end{proof}

Let $C_p$ be the branch of $\overline{C(F)}$ at $P$. Exactly as in Section 4 by using the Puiseux expansion we can show that if $d_1=d_2$ then $\overline{F}(C_P)$ is smooth at $\overline{F}(P)$ and if $d_1>d_2$ then  $2\delta(\overline{F}(C_P))_{\overline{F}(P)}=(d_1-1)(d_1-d_2-d)+(d_1+1-1)(d-1)=(d_1-1)(d_1-d_2-1)+(d-1)$, where $d=\gcd(d_1,d_2)$.

To proceed further we also need:

\begin{lem}\label{rowS}
If $F$ is a generic mapping then
$$\overline{f}(P_i)^{d_2}\overline{g}(P_j)^{d_1}\neq\overline{f}(P_j)^{d_2}\overline{g}(P_i)^{d_1}$$
for $i,j\in\{1,2,\ldots,2(d_1+d_2-1)\}$ and $i\neq j$.
\end{lem}

\begin{proof}
Consider the set $X=\{(p,q,F)\in \Gamma\times \Gamma\times\Omega_3(d_1,d_2):\ p\neq q,\ \tilde{J}(F)(p)=\tilde{J}(F)(q)=
\tilde{f}(p)^{d_2}\tilde{g}(q)^{d_1}-\tilde{f}(q)^{d_2}\tilde{g}(p)^{d_1}=0\}$. Similarly as in Lemma \ref{lemstyczne'} we will prove that $X$ has codimension $3$, so there is a dense open subset $U\subset \Omega_3(d_1,d_2)$ such that the projection from $X$ has empty fibers over $F\in U$.

Indeed, take $p=(1:0:0:0)$, $q=(0:0:1:0)$ and $Y:=\{(p,q)\}\times\Omega_3(d_1,d_2)$. It suffices to show that $X_0=X\cap Y$ has codimension $3$ in $Y$. Let $a_{ij}$ and $b_{ij}$ be the parameters in $\Omega_3(d_1,d_2)$ giving respectively the coefficients of $\tilde{f}$ at $x^{d_1-i-j}y^iz^j$ and of $\tilde{g}$ at $x^{d_2-i-j}y^iz^j$.

The three equations describing $X_0$ are 
$$w_1=d_1a_{0,0}b_{101}-d_2a_{1,0}b_{0,0}=0, w_2=d_2a_{1,d_1-1}b_{0,d_2}-d_1a_{0,d_1}b_{1,d_2-1}=0,$$   
$$w_3=a_{0,0}^{d_2}b_{0,d_2}^{d_1}-a_{0,d_1}^{d_2}b_{0,0}^{d_1}=0.$$ Note that 
$X_0\cap\{a_{0,0}=0\}\subset\{a_{0,0}=b_{0,0}=w_2=0\}\cup \{a_{0,0}=a_{0,1}=w_2=0\}$ has codimension $3$. Similarly $X_0\cap\{b_{0,0}=0\}$ and $X_0\cap\{a_{0,d_1}=0\}$ have codimension $3$, however outside the set $\{a_{0,0}=0\}\cup\{b_{0,0}=0\}\cup\{a_{0,d_1}=0\}$ the three equations are obviously independent. Thus $X_0$ has codimension $3$ in $X$.
\end{proof}

Now we are in a position to compute $\sum_{z\in (\overline{\Delta}\setminus \Delta)} \delta_z$. If $d_1=d_2$ then $\overline{\Delta}$ has exactly $2(d_1+d_2-1)$ smooth points at infinity and consequently $\sum_{z\in (\overline{\Delta}\setminus \Delta)} \delta_z=0$ (see the statement after Lemma \ref{lemstyczne'}). So assume $d_1>d_2$, then $\overline{\Delta}$ has only one point at infinity $Q=(1:0:0)$.
In $Q$ the curve $\overline{\Delta}$ has exactly $r=2(d_1+d_2-1)$ branches $V_i=\overline{F}(C_{P_i})$. We have $2\delta(V_i)_Q=(d_1-1)(d_1-d_2-1)+(d-1)$. As in Section \ref{secDF} we have $V_i\cdot V_j=d_1(d_1-d_2)$. Consequently
$$\sum_i \delta(V_i)+ \sum_{i>j} V_i\cdot V_j=[(d_1-1)(d_1-d_2-1)+(d-1)](d_1+d_2-1)+$$
$$d_1(d_1-d_2)(d_1+d_2-1)(2(d_1+d_2-1)-1)=$$
$$2d_1(d_1-d_2)(d_1+d_2-1)^2+(-2d_1+d_2+d)(d_1+d_2-1).$$
\end{proof}

We can now prove the following:

\begin{theo}
There is a Zariski open, dense subset $U\subset \Omega_S(d_1,d_2)$
such that for every mapping $F\in U$ the discriminant $\Delta(F)=F(C(F))$ has only cusps and nodes as singularities.
The number of cusps is equal to
$$c(F)=2(d_1^2+d_2^2+3d_1d_2-3d_1-3d_2+1)$$
and the number of nodes is equal to
$$d(F)=(2d_1d_2-3)D^2-D(d_1+d_2+d-2)-2(d_1d_2-d_1-d_2),$$
where $D=d_1+d_2-1$ and $d=\gcd(d_1,d_2).$
\end{theo}

\begin{proof}
Let $d_1\geq d_2$ and $D=(d_1+d_2-1)$. By Lemma \ref{lemdegdisc'} we have $\deg\Delta(F)= 2d_1D$. Since $\Delta(F)$ is birational with $C(F)$ it has genus $g=D(D-2)+1$. Moreover, $\Delta(F)$ has only cusps and nodes as singularities thus by Theorem \ref{thmgenusdelta} we have
$$ \frac{1}{2}(2d_1D-1)(2d_1D-2)=D(D-2)+1+c(F)+d(F)+ \sum_{z\in (\overline{\Delta}\setminus \Delta)} \delta_z.$$
 Substituting
$$\sum_{z\in (\overline{\Delta}\setminus \Delta)} \delta_z=2d_1(d_1-d_2)D^2+(-2d_1+d_2+d)D$$
we obtain
$$c(F)+d(F)=(2d_1d_2-1)D^2-D(d_1+d_2+d-2).$$
Thus by Theorem \ref{thmcusps'} we get:
$$d(F)=(2d_1d_2-1)D^2-D(d_1+d_2+d-2)-2(d_1^2+d_2^2+3d_1d_2-3d_1-3d_2+1)=$$
$$(2d_1d_2-3)D^2-D(d_1+d_2+d-2)-2(d_1d_2-d_1-d_2).$$
\end{proof}

\begin{re}
{\rm If $d_1=d_2=d$ then the discriminant has $4d-2$ smooth points at infinity and in each of these points it is tangent to the line $L_\infty$ (at infinity) with multiplicity $d.$
If $d_1>d_2$ then the discriminant has only one point at infinity with $2(d_1+d_2-1)$ branches $V_1,\ldots, V_{2(d_1+d_2-1)}$ and each of these branches has delta invariant
$$\delta(V_i)=\frac{(d_1-1)(d_1-d_2-1)+(d-1)}{2}$$ and $V_i\cdot L_\infty=d_1$. Additionally $V_i\cdot V_j=d_1(d_1-d_2)$. In particular branches $V_i$ are smooth if and only if
$d_1=d_2$ or $d_1=d_2+1.$}
\end{re}

\section{Generalized cusps}\label{secGC}

In this section our aim is to estimate the number of cusps of
non-generic mappings. We start from:

\begin{defi}\label{dfGenCus}
Let $F \colon (\C^2,a)\to (\C^2,F(a))$ be a germ of a holomorphic mapping. We say
that $F$ has a generalized cusp at $a$ if  $F_a$ is proper, the
curve $J(F)=0$ is reduced near $a$ and  the discriminant of $F_a$
is not smooth at $F(a)$.
\end{defi}

\begin{re}
If $F_a$ is proper, $J(F)=0$ is reduced near $a$ and $J(F)$ is
singular at $a$ then it follows from Theorem 1.14 from \cite{jel} that
also the discriminant of $F_a$ is singular at $F(a)$ and hence
$F$ has a generalized cusp at $a$.
\end{re}

Now we introduce the index of generalized cusp:

\begin{defi}\label{dfGenCusIn}
Let $F=(f,g)\colon (\C^2,a)\to (\C^2,F(a))$ be a germ of a holomorphic mapping.
Assume that $F$ has a generalized cusp at a point $a\in \C^2$.
Since the curve $J(F)=0$ is reduced near $a$, we have that the set
$\{\nabla f=0\}\cap \{\nabla g=0\}$ has only isolated points near
$a.$ For a generic linear mapping $T\in GL(2)$,  if
$F'=(f',g')=T\circ F$ then $\nabla f'$ does not vanish
identically on any branch of $\{J(F)=0\}$ near $a$.  We say that
the  cusp of $F$ at $a$ has an index $\mu_a:={\dim}_\C {\mathcal
O}_a/(J(F'), J_{1,1}(F'))-{\dim}_\C {\mathcal O}_a/(f'_x, f'_y)$.
\end{defi}

\begin{re}
We show below that the index $\mu_a$ is well-defined and
finite. Moreover, it is easy to see that a simple cusp has index
one.
\end{re}

\begin{re}
Using the exact sequence $1.7$ from \cite{gm1} we see that
$$\mu_a={\dim}_\C {\mathcal O}_a/(J(F), J_{1,1}(F), J_{1,2}(F)).$$
Hence our index coincides with the classical local number of cusps
defined e.g. in \cite{gm1}.
\end{re}

We have (compare with  \cite{gm1}, \cite{gm2}, \cite{gaf}):

\begin{theo}\label{tw}
Let $X\subset \C^m$ be a smooth surface.
Let $F=(f,g)\in \Omega_m(d_1,d_2)$. Assume that $F|_X$ has a
generalized cusp at $a\in X$. If $U_a\subset X$ is a sufficiently
small ball around $a$ then $\mu_a$ is equal to the number of
simple cusps in $U_a$ of a  mapping $F'|_U$ where $F'\in \Omega_m(d_1',
d_2')$ is a generic mapping,  which is sufficiently
close to $F$ in the natural topology of $\Omega_m(d_1', d_2')$. Here $d_1'\ge d_1, d_2'\ge d_2$.
\end{theo}

\begin{proof}
We can assume that $X=\C^2$ and  $\nabla f$ does not vanish identically on any
branch of $\{J(F)=0\}$ near $a.$ In particular we have ${\rm dim}\
{\mathcal O_a}/(f_x,f_y)= {\rm dim } \ {\mathcal O_a}/(J(F), f_x,
f_y)<\infty.$

Let $F_i=(f_i, g_i)\in\Omega_2(d_1', d_2')$ be a sequence of
generic mappings, which is convergent to $F.$ Consider the
mappings $\Phi=(J(F), J_{1,1}(F))$, $\Phi_i=(J(F_i),
J_{1,1}(F_i))$, $\Psi=(\nabla f)$ and $\Psi_i=(\nabla f_i).$ Thus
$\Phi_i\to \Phi$ and $\Psi_i\to \Psi.$

Since $a$ is a cusp of $F$ we have $\Phi(a)=0$. Moreover
$d_a(\Phi)<\infty,$ where $d_a(\Phi)$ denotes the local
topological degree of $\Phi$ at $a$. Indeed, if $J_{1,1}(F)=0$ on
some branch $B$ of the curve $J(F)=0$ then the rank of $F_{|B}$
would be zero and by Sard theorem $F$ has to contract $B$, which
is a contradiction ($F_a$ is proper). By the Rouche Theorem (see \cite{cir}, p. 86), we
have that for large $i$ the mapping $\Phi_i$ has exactly
$d_a(\Phi)$ zeroes in $U_a$ and $\Psi_i$ has exactly
$d_a(\Psi)$ zeroes in $U_a$ (counted with multiplicities, if
$\Psi(a)\not=0$ we put $d_a(\Psi)=0$). However, the mappings $F_i$
are generic, in particular all zeroes of $\Phi_i$ and $\Psi_i$ are
simple. Moreover the zeroes of $\Phi_i$ which are not cusps of $F_i$ are
zeroes of $\Psi_i$.
Hence $\mu_a=d_a(\Phi)-d_a(\Psi)$ is indeed the number of
simple cusps of $F_i$ in $U_a$.
\end{proof}

\begin{co}
Let $X$ be a smooth affine surface. If $F=(f, g)\colon X\to\C^2$ is an arbitrary polynomial mapping
with $\deg f\le d_1$, $\deg g\le d_2$  and generalized cusps at points $a_1,\ldots, a_r$ then $\sum^r_{i=1}
\mu_{a_i}\le c_X(d_1,d_2)$, where $c_X(d_1,d_2)$ is the number of cusps of a generic mapping from $\Omega_X(d_1,d_2).$
\end{co}

For example we have:

\begin{co}\label{cor}
Let $F\in \Omega_2(d_1,d_2)$. Assume that $F$ has generalized
cusps at points $a_1,\ldots, a_r$. Then $\sum^r_{i=1} \mu_{a_i}\le d_1^2+d_2^2+3d_1d_2-6d_1-6d_2+7$.
In particular the numbers of singular germs $\{F_a, \ a\in \C^2\}$ which are finitely determined and are not folds, is bounded
by the number $d_1^2+d_2^2+3d_1d_2-6d_1-6d_2+7 $.
\end{co}

\begin{proof}
Let $F_a$ be a singular germ which is finitely determined. Then the curve $J(F_a)$ is reduced. There are two possibilities:

1) the point $F(a)$ is a non-singular point of $\Delta(F),$

2) the point $F(a)$ is a singular point of $\Delta(F).$

\noindent In the case 1) we have by \cite{jel} that $F_a$ is equivalent to the germ $(x,y)\to(x^k,y)$ and since $J(F_a)$ is reduced we have $k=2$, i.e. $F_a$ is a fold.

\noindent In the case 2) $F_a$ is a generalized cusp. Hence the number of  germs  $F_a$ which are finitely determined and are not folds is bounded by the number of generalized cusps. It follows directly from Theorem \ref{tw} that the latter number is bounded by  the number of cusps of a generic mapping from $\Omega_2(d_1,d_2)$.
\end{proof}

\begin{re}
{\rm In the same way we can show that for the mapping $F\in \Omega_S(d_1,d_2)$ the numbers of singular germs $\{F_a, \ a\in S \}$ which are finitely determined and are not folds, is bounded
by the number  $2(d_1^2+d_2^2+3d_1d_2-3d_1-3d_2+1)$.}
\end{re}

\section{Proper deformations}\label{finite}
In previous sections we considered the family $\Omega_X(d_1,\ldots,d_m)$, of course we can consider also other families of polynomial mappings and try to investigate their properties. Let $\mathcal F$ be any algebraic family of generically-finite polynomial mappings $f_p: X\to \C^m; \ p\in {\mathcal F}$, where $X$ is a smooth irreducible affine variety. We would like to know the behavior of proper mappings in  such family. In general proper mappings do not form an algebraic subset of $\mathcal F$ but only constructible one. However we show that there is some regular behavior in such  family. We have:

\begin{theo}\label{rodzina}
Let $P, X, Y$ be  smooth irreducible affine algebraic varieties and let $F: P\times X\to P\times Y$ be a generically finite mapping.
The mapping $F$ induces a family ${\mathcal F}=\{ f_p(\cdot)=F(p,\cdot), \ p\in P \}$. Then either there exists a Zariski open dense subset 
$U\subset P$ such that for every $p\in U$ the mapping $f_p$ is proper,
or there exists a Zariski open dense subset 
$V\subset P$ such that for every $p\in V$ the mapping $f_p$ is not proper. 

Moreover, in the first case we have:

a) for every non-proper mapping $f_p$ in the family $\mathcal F$ we have  $\mu(f_p)<\mu(F)$, where $\mu(f)$ denotes the geometric degree of $f$,

b) generic mappings in $\mathcal F$ are topologically equivalent, i.e., there exists a Zariski open dense subset $W\subset P$ such that for every $p,q\in W$ the mappings $f_p$ and $f_q$ are topologically equivalent.
\end{theo}

\begin{proof}
First note that for every $(p,x)\in P\times X$ we have $\mu_{(p,x)}(F)= \mu_x (f_p)$ (here $\mu_x(f)$ denotes the local multiplicity of $f$ in $x$). In the sequel we use the fact that a mapping $g: X\to Y$ is proper
over a point $y\in Y$ if and only if $\sum_{g(x)=y} \mu_x(g)=\mu(g)$ (see \cite{Je}, \cite{Je1}).

Let $S$ be the non-properness set of $F$ (see e.g.  \cite{Je}, \cite{Je1}). If $S=\emptyset$, then all mappings $f_p$ are proper. 
Hence we can assume that $S\not= \emptyset$ and consequently $S$ is a hypersurface. Let $\pi: S\to P$ be the canonical projection. We have two possibilities:

\item{(1)} $\pi(S)$ is dense in $P$.

\item{(2)} $\pi(S)$ is not dense in $P$.

In the first case $\pi(S)$ is dense and constructible so a generic mapping $f_p$ is not proper. In the second case $S$ has dimension $\dim P+\dim X-1$ and a fiber of $\pi$ has dimension at most dim $X$. This immediately implies that the set $\overline{\pi(S)}$ is a hypersurface in $P$. Moreover, fibers of $\pi$ are the whole space $X$. This means that for all $p\in\pi(S)$ we have $\mu(f_p)<\mu(F)$. Of course outside $\pi(S)$ the mappings $f_p$ are proper. Two generic mappings are topologically equivalent by \cite{jel2}, Theorem 4.3.
\end{proof}

Now we state the main result of this section:

\begin{theo}
Let $X\subset \C^n$ be a smooth irreducible affine variety of dimension $k$ and let $F: X\to \C^m$ be a polynomial mapping.
If $m\ge k$, then there exists a Zariski open dense subset $U$ in the space of linear mappings ${\mathcal L}(\C^n,\C^m)$ 
such that:

a) for every $L\in U$ the mapping $F+L$ is a finite mapping. 

b) for all $L\in U$ the mappings $F+L$ are topologically equivalent.

c) for all $L\in U$ the mappings $F+L$ have only generic singularities,i.e., $F$ is transversal to the Thom-Boardman strata.
\end{theo}

\begin{proof}
Let $G: X\ni x \mapsto (x, F(x))\in X\times \C^m$ and $\tilde{X}=graph(G)\cong X.$ Since $m\ge \dim \tilde{X}$  a generic linear projection 
$\pi : \tilde{X}\to \C^m$ is a proper mapping. Hence also the mapping $\pi\circ G$ is proper.
Consequently for a general matrix $A\in GL(m,m)$ and general linear mapping $L\in {\mathcal L}(\C^n,\C^m)$ the mapping $H(A,L)=A(F_1,\ldots,F_m)^T+L$ is proper. Hence also the mapping $A^{-1}\circ H(A,L)$ is proper. This means that the mapping $F+ A^{-1}L$ is proper. But we can specialize the matrix $A$ to the identity and the mapping $L$ to a given linear mapping $L_0\in {\mathcal L}(\C^n,\C^m)$. Hence we see that there is a dense subset of linear mappings $L\in {\mathcal L}(\C^n,\C^m)$ such that the mapping $F+L: X\to \C^m$ is proper. Consider the algebraic family ${\mathcal F}=\{ F+L, L\in {\mathcal L}(\C^n,\C^m)\}$. By Theorem
\ref{rodzina} there exists a Zariski dense open subset $U\subset {\mathcal L}(\C^n,\C^m)$ such that every mapping $F+L; \ L\in U$ is proper and  all these mappings are topologically equivalent. Statement c) follows from Corollary \ref{wazne}.
\end{proof}

In particular for a given mapping $F:\C^2\to\C^2$ we can consider the ``linear'' deformation $F_L=F+L; \ L\in {\mathcal L}(\C^2,\C^2).$
A general member of this deformation is locally stable and proper. If $F$ is not $\mathcal{A}$ finitely determined, then this deformation gives in general a different number of cusps and folds than a ``generic'' deformation considered in this paper. We give here an example of a finitely ${\mathcal K}$ determined germ $F$ which has at least two non-equivalent stable deformations.

\begin{ex}
{\rm Take $F(x,y)=(x,y^3).$ This germ is finitely ${\mathcal K}$ determined, because it is ${\mathcal K}$ equivalent to the cusp $G(x,y)=(x, y^3+xy)$, which is stable. Consider two deformations of $F$: the first one linear $F_t=(x, y^3+ty)$ and the second one given by $G_t(x,y)=(x, y^3+txy)$. The members of the first family do not have a cusp at all and the members of the second family have exactly one cusp at $0$.}
\end{ex}

This means that (contrary to the case of $\mathcal A$ finitely determined germs) we can not define the numbers $c(F)$ and $d(F)$ for $F$ using stable deformations.

    \end{document}